\documentclass[10pt]{article}
\usepackage{defs}
\usepackage{subcaption}

\pgfplotsset{
  table/search path={figs},
}

\hypersetup{
  pdftitle={The Indefinite Proximal Gradient Method},
  pdfauthor={Geoffroy Leconte and Dominique Orban},
  pdfsubject={Nonsmooth Regularized Optimization},
  pdfkeywords={},
}

\title{The Indefinite Proximal Gradient Method}
\author{%
  Geoffroy Leconte\footnote{%
    GERAD and Department of Mathematics and Industrial Engineering, Polytechnique Montr\'eal. E-mail: \href{mailto:geoffroy.leconte@polymtl.ca}{geoffroy.leconte@polymtl.ca}.
  }
  \and
  Dominique Orban\footnote{%
    GERAD and Department of Mathematics and Industrial Engineering, Polytechnique Montr\'eal. E-mail: \href{mailto:dominique.orban@gerad.ca}{dominique.orban@gerad.ca}.
  }
  \thanks{Research supported by an NSERC Discovery grant.}
}

\dedicatory{In honor and memory of our colleague and friend Daniela di Serafino.}

\begin{document}
    \maketitle
    \thispagestyle{mytitlepage}

    \begin{abstract}
      We introduce a variant of the proximal gradient method in which the quadratic term is diagonal but may be indefinite, and is safeguarded by a trust region.
      Our method is a special case of the proximal quasi-Newton trust-region method of \citet{aravkin-baraldi-orban-2022}.
      We provide closed-form solution of the step computation in certain cases where the nonsmooth term is separable and the trust region is defined in the infinity norm, so that no iterative subproblem solver is required.
      Our analysis expands upon that of \citep{aravkin-baraldi-orban-2022} by generalizing the trust-region approach to problems with bound constraints.
      We provide an efficient open-source implementation of our method, named TRDH, in the Julia language in which Hessians approximations are given by diagonal quasi-Newton updates.
      TRDH evaluates one standard proximal operator and one indefinite proximal operator per iteration.
      We also analyze and implement a variant named iTRDH that performs a single indefinite proximal operator evaluation per iteration.
      We establish that iTRDH enjoys the same asymptotic worst-case iteration complexity as TRDH.
      We report numerical experience on unconstrained and bound-constrained problems, where TRDH and iTRDH are used both as standalone and subproblem solvers.
      Our results illustrate that, as standalone solvers, TRDH and iTRDH improve upon the quadratic regularization method R2 of \citep{aravkin-baraldi-orban-2022} but also sometimes upon their quasi-Newton trust-region method, referred to here as TR-R2, in terms of smooth objective value and gradient evaluations.
      On challenging nonnegative matrix factorization, binary classification and data fitting problems, TRDH and iTRDH used as subproblem solvers inside TR improve upon TR-R2 for at least one choice of diagonal approximation.
    \end{abstract}


    \pagestyle{myheadings}

    \section{Introduction}

    We consider the nonsmooth regularized problem
    \begin{equation}%
      \label{eq:nlp}
      \minimize{x \in \R^n} \ f(x) + h(x)
      \quad \st \ \ell \leq x \leq u,
    \end{equation}
    where \(\ell \in (\R \cup \{-\infty\})^n\), \(u \in (\R \cup \{+\infty\})^n\) with \(\ell \leq u\) componentwise, \(f: \R^n \to \R\) is continuously differentiable on an open set containing \([\ell, u]\), and \(h: \R^n \to \R \cup \{+\infty\}\) is proper and lower semicontinuous (lsc).
    A component \(\ell_i = -\infty\) or \(u_i = +\infty\) indicates that \(x_i\) is unbounded below or above, respectively.
    Both \(f\) and \(h\) may be nonconvex.
    Typically, \(h\) is nonsmooth and serves to identify a local minimizer of \(f\) with desirable features, such as sparsity.

    Numerical methods for~\eqref{eq:nlp} are typically based on the proximal-gradient method \citep{lions-mercier-1979}.
    \citet{aravkin-baraldi-orban-2022} provide an overview of recent works focusing on~\eqref{eq:nlp} where both \(f\) and \(h\) may be nonconvex, to which we refer the reader.
    In addition, they propose two methods: an adaptive quadratic regularization approach named R2, which may be viewed as a proximal-gradient method with adaptive step size, and a quasi-Newton trust-region method named TR in which the subproblem consist in minimizing a quadratic approximation of \(f\) about the current iterate regularized by a model of \(h\) inside a trust region.
    Typically, an explicit solution to the subproblem is not known, and an iterative procedure must be used---the authors use R2.
    \citet{aravkin-baraldi-orban-2022b} develop a similar approach designed for applications where \(f\) is a least-squares residual.
    They propose a trust-region and a regularization approach.
    Again, they use R2 as subproblem solver.

    In both R2 and the traditional proximal-gradient method, a uniformly positively scaled \(\nabla f\) is used to compute a step, and that computation relies on the proximal operator---see \Cref{sec:background} for precise definitions.
    Some authors consider positive-definite diagonal gradient scaling---see below.
    In the present research, we contend that generalizing the definition of the proximal operator by allowing a diagonal scaling of \(\nabla f\), and even permitting negative scaling factors, continues to allow us to derive analytical solutions for several nonsmooth terms of interest in applications.
    Moreover, such potentially indefinite scaling might allow the model to better capture inherent nonconvexity in \(f\) about the current iterate.
    We devise a trust-region method based on those ideas and name it TRDH, which stands for \emph{trust-region method with diagonal Hessian approximations}.
    At every iteration, our method performs the evaluation of both a classical proximal operator and a generalized proximal operator with indefinite diagonal scaling.
    However, it is possible to modify TRDH slightly to devise a variant that requires a single generalized proximal operator evaluation per iteration while preserving the asymptotic worst-case evaluation complexity bound.
    We name the variant iTRDH, which stands for \emph{indefinite trust-region method with diagonal Hessian}.

    Diagonal gradient scaling could be referred to as diagonal quasi-Newton, and though the literature appears to be thin on that subject, there exist a few references.
    Diagonal Hessian approximations range from a multiple of the identity, as in the traditional proximal-gradient or the spectral gradient method, to approximations computed based on a relaxed secant equation.
    We go into more details in \Cref{sec:quasi-Newton}.

    Because we use the \(\ell_\infty\)-norm to define the trust region, we are able to treat bound constraints naturally, by taking the intersection of \([\ell, \, u]\) with the trust region.
    Our generalized \emph{indefinite} proximal operators must take the indicator of the resulting box into account together with \(h\), or a model thereof.

    A by-product of the present research is an efficient software implementation of TRDH and iTRDH, both as standalone solvers, and as subproblem solvers for use inside TR\@.
    TRDH and iTRDH can use several diagonal Hessian approximations: the multiple of the identity given by the spectral gradient approximation, and two diagonal quasi-Newton approximations based on the weak secant equation.
    Our open-source implementations in Julia are available from \url{https://github.com/JuliaSmoothOptimizers/RegularizedOptimization.jl}.

    We report numerical experience conducted with TRDH, iTRDH, TR-TRDH and TR-iTRDH on a range of unconstrained and bound-constrained problems, where TR-TRDH and TR-iTRDH refer to TRDH and iTRDH used as subsolvers in TR.
    Our general conclusion is that for at least one choice of a diagonal Hessian approximation, TRDH and iTRDH outperform R2 in terms of evaluations of \(f\) and \(\nabla f\).
    In addition, for at least one choice of a diagonal Hessian approximation, TR-TRDH and TR-iTRDH outperform TR-R2 on the same metric.
    Our detailed results are in \Cref{sec:numerical}.

    \subsection*{Related research}

    Most of the literature focuses on positive-definite diagonal scaling of the proximal operator;
    \citet{becker-fadili-2012} and \citet{becker-fadili-ochs-2019} consider positive-definite quasi-Newton approximations of the form \(H = D + V V^T\) where \(D\) is positive definite and diagonal, and a specialized procedure to solve the proximal quasi-Newton subproblems.
    Under their assumptions, \(H\) remains uniformly bounded.
    Their analysis is restricted to \(f\) and \(h\) convex and does not provide complexity bounds.

    \citet{duchi-hazan-2011} present the ADAGRAD algorithm, which is commonly used in online learning.
    ADAGRAD is a stochastic algorithm but a deterministic implementation of it would minimize a smooth objective such as \(f\) using a variant of the projected gradient algorithm, that can be seen as a special case of the proximal gradient method, using a positive diagonal gradient scaling.
    At iteration \(k\), the scaling is set to \(\diag (G_k)^{-1/2}\) where \(G_k = \sum_{j=0}^k \nabla f(x_j) \nabla f(x_j)^T\).

    \citet{scheinberg-tang-2016} also focus on positive-definite and uniformly-bounded Hessian approximations for \(f\) and \(h\) convex.
    Their numerical results employ a limited-memory BFGS approximation, although the latter is liable to grow unbounded \citep[\S\(8.4\)]{conn-gould-toint-2000}.
        The proximal quasi-Newton subproblem is solved inexactly with a coordinate descent algorithm.

    \citet{aravkin-baraldi-orban-2022} propose a proximal quasi-Newton trust-region method for nonconvex \(f\) and \(h\) under weak assumptions, accompanied by a complexity analysis.
    They employ limited-memory BFGS or SR1 approximations and solve subproblems using the proximal-gradient method or R2.
    Though their analysis assumes the Hessian approximations are uniformly bounded, there are known procedures to estimate bounds on the eigenvalues of quasi-Newton approximations after an update, and those bounds can be used to skip or modify the update to keep it bounded \citep{lotfi-bonniot-orban-lodi-2020}.

    Similarly, \citet{baraldi-kouri-2022} present a proximal trust-region method with inexact objective and gradient evaluations, under the additional assumption that \(h\) must be convex.
    They show numerical results with exact Hessian when it is available, and employ iterative solvers to approximate it otherwise.

    \subsection*{Notation}

    \(\B\) is the unit ball centered at the origin and defined by a norm dictated by the context, and \(\Delta \B\) is the ball of radius \(\Delta > 0\) centered at the origin.
    For fixed \(x \in \R^n\), the ball of radius \(\Delta\) centered at \(x\) is \(x + \Delta \B\).
    When it is necessary to indicate that \(\B\) is defined by the \(\ell_p\)-norm, for \(1 \leq p \leq +\infty\), we write \(\B_p\).
    For \(A \subseteq \R^n\), the indicator of \(A\) is \(\chi(\cdot \mid A): \R^n \to \R \cup \{+\infty\}\) defined as \(\chi(x \mid A) = 0\) if \(x \in A\) and \(+\infty\) otherwise.
    If \(A \neq \varnothing\), \(\chi(\cdot \mid A)\) is proper.
    If \(A\) is closed, \(\chi(\cdot \mid A)\) is lsc.
    For \(D \in \R^{m \times n}\) with elements \(\delta_{i,j}\), \(|D| \in \R^{m \times n}\) has elements \(|\delta_{i,j}|\).
    For a finite set \(A \subset \N\), we denote \(|A|\) its cardinality.
    If \(f_1\) and \(f_2\) are two positive functions of \(\epsilon > 0\), we say that \(f_1(\epsilon) = O(f_2(\epsilon))\) if there exists a constant \(C > 0\) such that \(f_1(\epsilon) \leq C f_2(\epsilon)\) for all \(\epsilon > 0\) sufficiently small.

    \section{Background}%
    \label{sec:background}

    The following are standard variational analysis concepts---see, e.g., \citep{rockafellar-wets-1998}.
    Let \(\phi: \R^n \to \widebar{\R}\) and \(\bar{x} \in \R^n\) where \(\phi\) is finite.
    The Fr\'echet subdifferential of \(\phi\) at \(\bar{x}\) is the closed convex set \(\widehat{\partial} \phi(\bar{x})\) of \(v \in \R^n\) such that
    \[
      \liminf_{\substack{x \to \bar{x} \\ x \neq \bar{x}}} \frac{\phi(x) - \phi(\bar{x}) - v^T (x - \bar{x})}{\|x - \bar{x}\|} \geq 0.
    \]
    The limiting subdifferential of \(\phi\) at \(\bar{x}\) is the closed, but not necessarily convex, set \(\partial \phi(\bar{x})\) of \(v \in \R^n\) for which there exist \(\{x_k\} \to \bar{x}\) and \(\{v_k\} \to v\) such that \(\{\phi(x_k)\} \to \phi(\bar{x})\) and \(v_k \in \widehat{\partial} \phi(x_k)\) for all \(k\).
    \(\widehat{\partial} \phi(\bar{x}) \subset \partial \phi(\bar{x})\) always holds.

    The horizon subdifferential of \(\phi\) at \(\bar{x}\) is the closed, but not necessarily convex, cone \(\partial^{\infty} \phi(\bar{x})\) of \(v \in \R^n\) for which there exist \(\{x_k\} \to \bar{x}\), \(\{v_k\}\) and \(\{\lambda_k\} \downarrow 0\) such that \(\{\phi(x_k)\} \to \phi(\bar{x})\), \(v_k \in \widehat{\partial} \phi(x_k)\) for all \(k\), and \(\{\lambda_k v_k\} \to v\).

    If \(C \subseteq \R^n\) and \(\bar{x} \in C\), the closed convex cone \(\widehat{N}_C(\bar{x}) := \widehat{\partial} \chi(\bar{x} \mid C)\) is called the regular normal cone to \(C\) at \(\bar{x}\).
    The closed cone \(N_C(\bar{x}) := \partial \chi(\bar{x} \mid C) = \partial^{\infty} \chi(\bar{x} \mid C)\) is called the normal cone to \(C\) at \(\bar{x}\).
    \(\widehat{N}_C(\bar{x}) \subseteq N_C(\bar{x})\) always holds, and is an equality if \(C\) is convex.

    If \(\phi: \R^n \to \widebar{\R}\) is proper lsc, and \(C \subseteq \R^n\) is closed, we say that the \emph{constraint qualification} is satisfied at \(\bar{x} \in C\) for the constrained problem
    \[
      \minimize{x \in \R^n} \ \phi(x) \quad \st \ x \in C
    \]
    if
    \begin{equation}%
      \label{eq:cq}
      \partial^{\infty} \phi(\bar{x}) \cap N_C(\bar{x}) = \{0\}.
    \end{equation}
    As an example where~\eqref{eq:cq} fails to hold, let \(n = 1\), \(f(x) = 0\) for all \(x \in \R\),
    \[
      h(x) =
      \begin{cases}
        x & \text{if } x \geq 0, \\
        +\infty & \text{if } x < 0,
      \end{cases}
    \]
    and \(C = [0, \, 1]\).
    Clearly, \(\bar{x} = 0\) is the only stationary point and \(N_C(\bar{x}) = \{v \mid v \leq 0\}\).
    Let \(\epi h\) denote the epigraph of \(h\), i.e., the set \(\{(x, t) \mid x \in \R^n, \ t \geq h(x)\}\).
    Because \(h\) is locally lsc about \(\bar{x}\), \citep[Theorem~\(8.9\)]{rockafellar-wets-1998} yields
    \[
      \partial^{\infty} h(\bar{x}) = \{ v \mid (v, 0) \in N_{\epi h}(\bar{x}, h(\bar{x}))\}
      = \{ v \mid (v, 0) \in N_{\epi h}(0, 0)\}
      = N_C(\bar{x}),
    \]
    and the constraint qualification does not hold.
    The reason is that \(h\) and \(\chi(\cdot \mid C)\) both have a jump discontinuity at \(\bar{x}\) and take the value \(+\infty\) for \(x < \bar{x}\), which is akin to repeating a constraint.

    We say that \(\bar{x}\) is first-order stationary for~\eqref{eq:nlp} if \(0 \in \partial(f + h + \chi(\cdot \mid [\ell, u]))(\bar{x}) = \nabla f(\bar{x}) + \partial (h + \chi(\cdot \mid [\ell, u]))(\bar{x})\).
    If \(\bar{x}\) is a local solution of~\eqref{eq:nlp}, it is first-order stationary \citep[Theorem~\(10.1\)]{rockafellar-wets-1998}.
    Under~\eqref{eq:cq} applied to~\eqref{eq:nlp}, which reads \((\nabla f(\bar{x}) + \partial^{\infty} h(\bar{x})) \cap N_{[\ell, u]}(\bar{x}) = \{0\}\), the necessary optimality condition can be written equivalently as \(0 \in \nabla f(\bar{x}) + \partial h(\bar{x}) + N_{[\ell, u]}(\bar{x})\) \citep[Theorem~\(8.15\)]{rockafellar-wets-1998}.
    When~\eqref{eq:nlp} is unconstrained or \(\bar{x} \in \interior [\ell, u]\), the normal cone is \(\{0\}\), the constraint qualification is satisfied, and the necessary condition reduces to \(0 \in \nabla f(\bar{x}) + \partial h(\bar{x})\).

    The proximal operator associated with \(h\) is
    \begin{equation}%
      \label{eq:def-prox}
      \prox_{\nu h}(q) := \argmin{x} \ \tfrac{1}{2} \nu^{-1} \|x - q\|_2^2 + h(x),
    \end{equation}
    where \(\nu > 0\) is a preset steplength.
    If \(h\) is prox-bounded and \(\nu > 0\) is sufficiently small, \(\prox_{\nu h}(q)\) is a nonempty and closed set.
    It may contain multiple elements.

    The proximal gradient method \citep{lions-mercier-1979} is a generalization of the gradient method that takes the nonsmooth term into account.
    It generates iterates \(\{x_k\}\) according to
    \begin{equation}%
      \label{eq:pg-iter-basic}
      x_{k+1} \in \prox_{\nu_k h}(x_k - \nu_k \nabla f(x_k)).
    \end{equation}
    Equivalently, and more instructively,
    \begin{subequations}%
      \label{eq:pg-iter}
      \begin{align}
        s_k & \in \argmin{s} \ \tfrac{1}{2} \nu_k^{-1} \|s + \nu_k \nabla f(x_k)\|_2^2 + h(x_k + s), \label{eq:pg-iter-s}
        \\  & = \argmin{s} \ \nabla f(x_k)^T s + \tfrac{1}{2} \nu_k^{-1} \|s\|_2^2 + h(x_k + s),   \label{eq:pg-iter-model}
        \\ x_{k+1} & = x_k + s_k,
      \end{align}
    \end{subequations}
    where it becomes clear that the computation of \(s_k\) results from formulating a simple quadratic model
    \begin{equation}%
      \label{eq:q-model}
      f(x_k) + \nabla f(x_k)^T s + \tfrac{1}{2} \nu_k^{-1} \|s\|_2^2 \approx
      f(x_k + s)
    \end{equation}
    of the smooth term \(f\), to which we add the nonsmooth term \(h\).

    The effectiveness of~\eqref{eq:pg-iter} comes from the observation that it is possible to determine a closed form solution \(s_k\), or at least a specialized procedure to identify such a global minimizer, for numerous choices of \(h\) that are relevant in applications \citep{beck-2017}, without resort to a general-purpose optimization solver.

    \citet{aravkin-baraldi-orban-2022} devise a trust-region quasi-Newton method for~\eqref{eq:nlp} without explicit bound constraints in which subproblems are solved inexactly by a method closely related to~\eqref{eq:pg-iter}.
    However, the subproblem has the additional trust-region constraint \(s \in \Delta_k \B\):
    \begin{equation}%
      \label{eq:pg-iter-s-tr}
      s_k \in \argmin{s} \ \nabla f(x_k)^T s + \tfrac{1}{2} \nu_k^{-1} \|s\|_2^2 + h(x_k + s) + \chi(s \mid \Delta_k \B),
    \end{equation}
    where \(\Delta_k > 0\), which may be construed as changing \(h(x_k + s)\) to \(h(x_k + s) + \chi(s \mid \Delta_k \B)\).
    Because both \(s \mapsto h(x_k + s)\) and \(s \mapsto \chi(s \mid \Delta \B)\) are proper and lsc, so is their sum.
    In addition, because \(\Delta_k \B\) is bounded, \(h\) no longer needs to be prox-bounded for a solution to exist.
    Again, a closed form solution may be found for a variety of cases of interest in applications \citep{aravkin-baraldi-orban-2022,aravkin-baraldi-orban-2022b}.

    \section{The indefinite proximal operator}

    \subsection{Definition and Examples}

    We restrict ourselves to the case where \(h\) is separable, i.e., \(h(x) = h_1(x_1) + \cdots + h_n(x_n)\), in which case the search for an element \(s\) of~\eqref{eq:q-model} decouples componentwise:
    \[
      s_i \in \argmin{s_i} \ g_i s_i + \tfrac{1}{2} \nu^{-1} s_i^2 + h_i(x_i + s_i),
    \]
    where \(g_i\) is the \(i\)-th component of \(\nabla f(x_k)\).
    If \(\B\) is defined in the \(\ell_\infty\)-norm,~\eqref{eq:pg-iter-s-tr} also decouples.
    Without loss of generality, and accounting for the bound constraints of~\eqref{eq:nlp}, we write it as
    \begin{equation}%
      \label{eq:prox-decoupled-tr}
      s_i \in \argmin{s_i} \ g_i s_i + \tfrac{1}{2} \nu^{-1} s_i^2 + h_i(x_i + s_i) + \chi(s_i \mid [\tilde{\ell}_i, \, \tilde{u}_i]),
    \end{equation}
    where \(\tilde{\ell}_i := \max(\ell_i - x_i, -\Delta)\) and \(\tilde{u}_i := \min(u_i - x_i, \Delta)\) are the \(i\)th components of \(\tilde{\ell}\) and \(\tilde{u} \in \R^n\), which describe the intersection of the box \([\ell, u]\) with the trust region.

    Replacing the quadratic coefficient in~\eqref{eq:prox-decoupled-tr} by a positive scalar \(\delta_i > 0\) does not complicate the problem as it simply amounts to changing the value of \(\nu\).
    In the following, we contend that replacing the quadratic coefficient by \(\delta_i \leq 0\) continues to allow us to derive a closed form solution in certain cases.
    Observe that modifying the quadratic term as described amounts to changing~\eqref{eq:q-model} to
    \[
      f(x_k) + \nabla f(x_k)^T s + \tfrac{1}{2} s^T D_k s,
    \]
    where \(D_k = \diag(d_k)\), and \(d_k \in \R^n\) has components \(\delta_{k.i}\).
    Crucially, we wish to allow \(D_k\) to be indefinite to capture any nonconvexity of \(f\) about \(x_k\) to a certain extent.
    While such change may make the quadratic model nonconvex, the indicator of \([\tilde{\ell}_i, \, \tilde{u}_i]\) ensures that a finite solution exists.

    \begin{definition}%
      \label{def:iprox}
      Let \(g \in \R^n\), \(D \in \R^{n \times n}\) be diagonal and \(C \subset \R^n\) be nonempty.
      The indefinite proximal operator of \(h\) with respect to \(g\), \(D\) and \(C\) is
      \begin{equation}%
        \label{eq:def-iprox}
        \iprox_{g, D, h, C}(q) := \argmin{x} \ g^T x + \tfrac{1}{2} x^T D x + h(x) + \chi(x \mid C).
      \end{equation}
    \end{definition}

    The usual proximal operator of \(\nu h\) at \(q \in \R^n\) is a special case of \Cref{def:iprox} in which \(C = \R^n\), \(D = \nu^{-1} I\), and \(g = -\nu^{-1} q\).

    \begin{example}%
      \label{ex:iprox-l0}
      Consider \(h(x) = \lambda \|x\|_0\), where \(\lambda > 0\), and \(C := [\ell, \, u]\), where \(\ell\), \(u \in \R^n\) with \(\ell < u\) componentwise.
      The computation of \(x\) in~\eqref{eq:def-iprox} decouples componentwise as
      \[
        x_i \in \argmin{x_i} g_i x_i + \tfrac{1}{2} \delta_i x_i^2 + \lambda |x_i|_0 + \chi(x_i \mid [\ell_i, \, u_i]).
      \]
      Note that
      \[
        g_i x_i + \tfrac{1}{2} \delta_i x_i^2 + \lambda |x_i|_0 =
        \begin{cases}
          0 & \text{if } x_i = 0 \\
          \lambda + g_i x_i + \tfrac{1}{2} \delta_i x_i^2 & \text{if } x_i \neq 0.
        \end{cases}
      \]
      If \(\delta_i = 0\), the solution set of the above problem is
      \begin{itemize}
        \item \([\ell_i, u_i]\) if \(g_i = 0\) and either \(\ell_i > 0\) or \(u_i < 0\);
        \item \(\argmin{} \{\lambda + g_i x_i \mid x_i \in M_i\}\) where \(M_i\) is the set \(\{\ell_i, u_i\}\) together with \(0\) if \(0 \in (\ell_i, u_i)\) in all other cases.
      \end{itemize}

      If \(\delta_i \neq 0\), the solution set is \(\argmin{} \{\lambda + g_i x_i + \tfrac{1}{2} \delta_i x_i^2 \mid x_i \in M_i\}\), where \(M_i\) is the finite set comprising \(\ell_i\), \(u_i\) together with
      \begin{itemize}
        \item \(-g_i / \delta_i\) if \(\delta_i > 0\);
        \item \(0\) if \(0 \in (\ell_i, u_i)\).
      \end{itemize}
      See \Cref{fig:iprox-l0} for an illustration of a few cases.
    \end{example}

    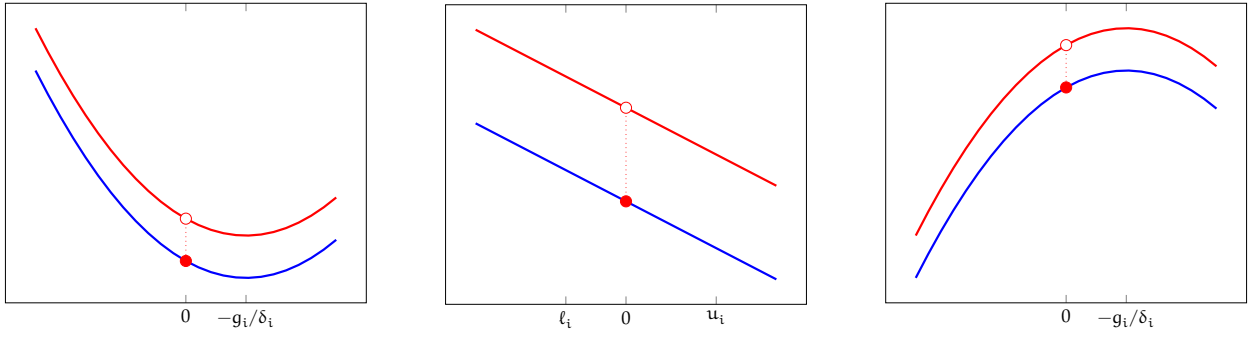
\begin{figure}[t]
      \centering
      \begin{tikzpicture}[scale=.7]
        \begin{axis}[
            xtick = {0, 2},
            xticklabels = {\(0\), \(-g_i / \delta_i\)},
            ytick = \empty,
          ]
          \addplot[blue, very thick] {(x - 2)^2 + 0.5};
          \addplot[red, very thick] {(x - 2)^2 + 10.5};
          \draw [red, dotted] (0, 14.5) -- (0, 4.5);
          \draw [red, fill=white] (0, 14.5) circle (3pt);
          \filldraw[red] (0, 4.5) circle (3pt);
        \end{axis}
      \end{tikzpicture}
      \hfill
      \begin{tikzpicture}[scale=.7]
        \begin{axis}[
            xtick = {-2, 0, 3},
            xticklabels = {\(\ell_i\), 0, \(u_i\)},
            ytick = \empty,
          ]
          \addplot[blue, very thick] {-x / 3};
          \addplot[red, very thick] {- x / 3 + 2};
          \draw [red, dotted] (0, 0) -- (0, 2);
          \draw [red, fill=white] (0, 2) circle (3pt);
          \filldraw[red] (0, 0) circle (3pt);
        \end{axis}
      \end{tikzpicture}
      \hfill
      \begin{tikzpicture}[scale=.7]
        \begin{axis}[
            xtick = {0, 2},
            xticklabels = {\(0\), \(-g_i / \delta_i\)},
            ytick = \empty,
          ]
          \addplot[blue, very thick] {-(x - 2)^2};
          \addplot[red, very thick] {-(x - 2)^2 + 10};
          \draw [red, dotted] (0, 6) -- (0, -4);
          \draw [red, fill=white] (0, 6) circle (3pt);
          \filldraw[red] (0, -4) circle (3pt);
        \end{axis}
      \end{tikzpicture}
      \caption{%
        \label{fig:iprox-l0}
        Illustration of \Cref{ex:iprox-l0} with \(\delta_i > 0\) (left), \(\delta_i = 0\) (center), and \(\delta_i < 0\) (right).
        The blue curve is \(g_i x_i + \tfrac{1}{2} \delta_i x_i^2\) while the red curve is \(g_i x_i + \tfrac{1}{2} \delta_i x_i^2 + \lambda |x_i|_0\).
        }
    \end{figure}

    \begin{example}%
      \label{ex:iprox-l1}
      Consider \(h(x) = \lambda \|x\|_1\) where \(\lambda > 0\), and \(C\) as in \Cref{ex:iprox-l0}.
      The computation of \(x\) in~\eqref{eq:def-iprox} decouples componentwise as
      \[
        x_i \in \argmin{x_i} g_i x_i + \tfrac{1}{2} \delta_i x_i^2 + \lambda |x_i| + \chi(x_i \mid [\ell_i, \, u_i]).
      \]
      Note that
      \[
        g_i x_i + \tfrac{1}{2} \delta_i x_i^2 + \lambda |x_i| =
        \begin{cases}
          (g_i + \lambda) x_i + \tfrac{1}{2} \delta_i x_i^2 & \text{if } x_i \geq 0 \\
          (g_i - \lambda) x_i + \tfrac{1}{2} \delta_i x_i^2 & \text{if } x_i < 0.
        \end{cases}
      \]

      If \(\delta_i = 0\), the solution of the above problem is
      \begin{itemize}
        \item \([\ell_i, \, \max(0, \ell_i)]\) if \(g_i = \lambda\);
        \item \([\min(u_i, 0), \, u_i]\) if \(g_i = -\lambda\);
        \item \(\proj(0 \mid [\ell_i, u_i])\) if \(0 \leq |g_i| < \lambda\);
        \item \(\{\ell_i\}\) if \(g_i > \lambda\);
        \item \(\{u_i\}\) if \(g_i < -\lambda\).
      \end{itemize}

      Consider now the case where \(\delta_i > 0\).
      The branch of quadratic for \(x_i \geq 0\) lies above the horizontal axis if \(g_i \geq -\lambda\), and that for \(x_i < 0\) does the same if \(g_i \leq \lambda\).
      Thus, if \(|g_i| \leq \lambda\), the unique unconstrained minimizer is \(0\) and, by convexity, the solution of the constrained problem is \(\proj(0 \mid [\ell_i, u_i])\).

      If \(g_i > \lambda\), for the unconstrained minimizer is \(\bar{x}_i := -(g_i - \lambda) / \delta_i < 0\).
      By convexity, the solution of the constrained problem is \(\proj(\bar{x}_i \mid [\ell_i, u_i])\).
      The situation is similar when \(g_i < -\lambda\).

      Consider finally the case where \(\delta_i < 0\).
      In this case, the unique constrained minimizer is $\argmin{} \{g_i x_i + \tfrac{1}{2} \delta_i x_i^2 + \lambda |x_i| \mid x_i \in M_i\}\) where \(M_i\) is the set containing \(\ell_i\) and \(u_i\) together with \(0\) if the latter lies inside \([\ell_i, u_i]\).

      \Cref{fig:iprox-l1} illustrates a few representative cases.
    \end{example}

    \begin{figure}[t]
      \centering
      \begin{tikzpicture}[scale=.7]
        \begin{axis}[
            xtick = {0},
            xticklabels = {0},
            ytick = \empty,
          ]
          \addplot[blue, very thick] {0.5 * x + abs(x)};
        \end{axis}
      \end{tikzpicture}
      \hfill
      \begin{tikzpicture}[scale=.7]
        \begin{axis}[
            xtick = {0},
            xticklabels = {0},
            ytick = \empty,
          ]
          \addplot[blue, very thick] {x + abs(x)};
        \end{axis}
      \end{tikzpicture}
      \hfill
      \begin{tikzpicture}[scale=.7]
        \begin{axis}[
            xtick = {0},
            xticklabels = {0},
            ytick = \empty,
          ]
          \addplot[blue, very thick] {2 * x + abs(x)};
        \end{axis}
      \end{tikzpicture}
      \\
      \begin{tikzpicture}[scale=.7]
        \begin{axis}[
            xtick = {0, 2},
            xticklabels = {0, \(\bar{x}_i\)},
            ytick = \empty,
          ]
          \addplot[blue, very thick] {-3 * x + 0.5 * x^2 + abs(x)};
        \end{axis}
      \end{tikzpicture}
      \hfill
      \begin{tikzpicture}[scale=.7]
        \begin{axis}[
            xtick = {0},
            xticklabels = {0},
            ytick = \empty,
          ]
          \addplot[blue, very thick] {0.5 * x + 0.5 * x^2 + abs(x)};
        \end{axis}
      \end{tikzpicture}
      \hfill
      \begin{tikzpicture}[scale=.7]
        \begin{axis}[
            xtick = {-2, 0},
            xticklabels = {\(\bar{x}_i\), 0},
            ytick = \empty,
          ]
          \addplot[blue, very thick] {3 * x + 0.5 * x^2 + abs(x)};
        \end{axis}
      \end{tikzpicture}
      \\
      \begin{tikzpicture}[scale=.7]
        \begin{axis}[
            xtick = {0},
            xticklabels = {0},
            ytick = \empty,
          ]
          \addplot[blue, very thick] {-x - 0.5 * x^2 + abs(x)};
        \end{axis}
      \end{tikzpicture}
      \hfill
      \begin{tikzpicture}[scale=.7]
        \begin{axis}[
            xtick = {0},
            xticklabels = {0},
            ytick = \empty,
          ]
          \addplot[blue, very thick] {0.2 * x - 0.5 * x^2 + abs(x)};
        \end{axis}
      \end{tikzpicture}
      \hfill
      \begin{tikzpicture}[scale=.7]
        \begin{axis}[
            xtick = {0},
            xticklabels = {0},
            ytick = \empty,
          ]
          \addplot[blue, very thick] {x - 0.5 * x^2 + abs(x)};
        \end{axis}
      \end{tikzpicture}
      \caption{%
        \label{fig:iprox-l1}
        Illustration of \Cref{ex:iprox-l1}.
        Top row: \(\delta_i = 0\).
        From left to right: \(0 \leq g_i < \lambda\), \(g_i = \lambda\), and \(g_i > \lambda\).
        Middle row: \(\delta_i > 0\).
        From left to right: \(g_i < -\lambda\), \(|g_i| \leq \lambda\), and \(g_i > \lambda\).
        Bottom row: \(\delta_i < 0\).
        From left to right: \(g_i \leq -\lambda\), \(|g_i| < \lambda\), and \(g_i > \lambda\).
        }
    \end{figure}
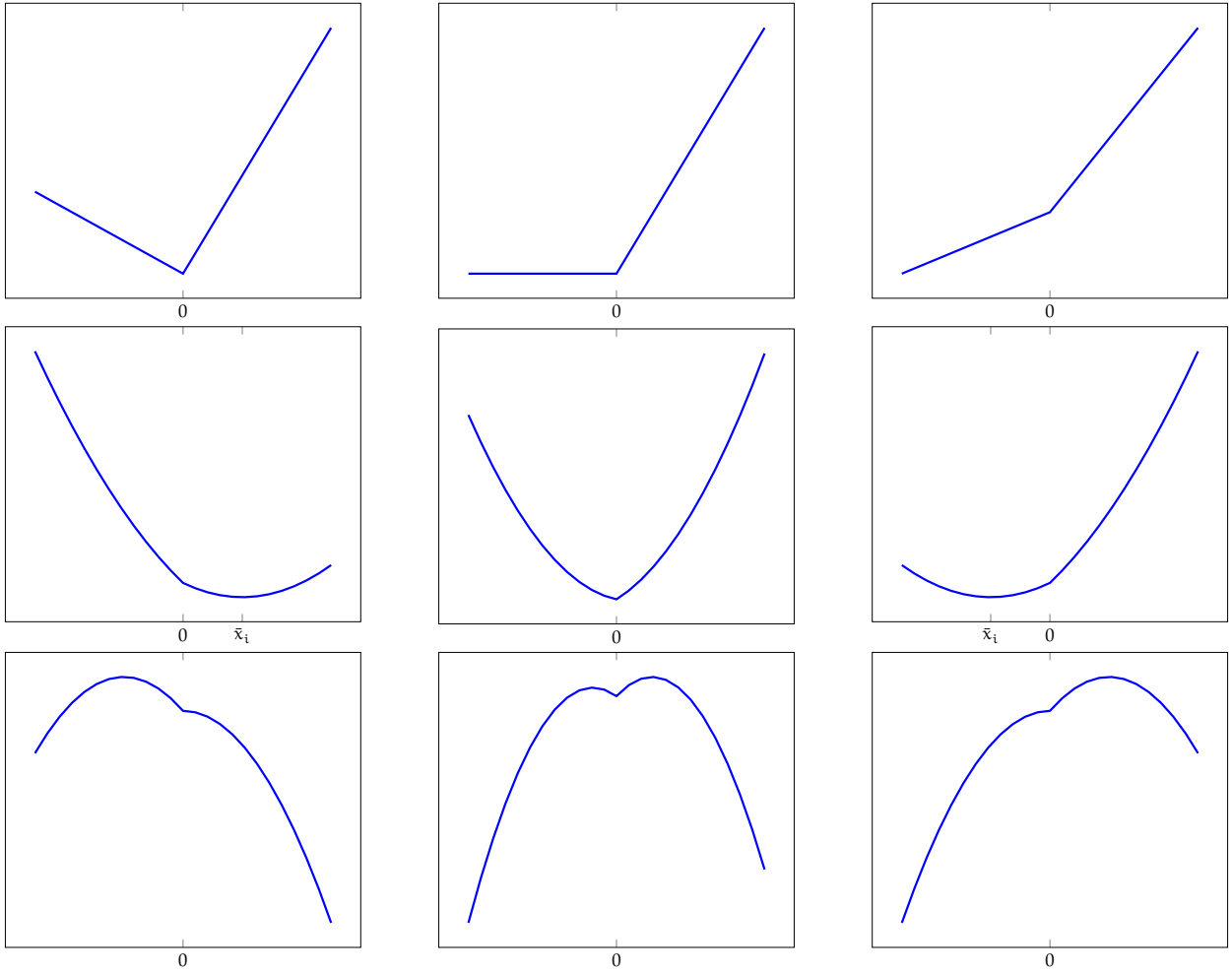

    \subsection{Properties}

    At \(x \in \R^n\) where \(h\) is finite, consider models
    \begin{subequations}%
      \label{eq:def-models}
      \begin{align}
        \varphi(s; x) & \phantom{:}\approx f(x + s) \label{eq:def-varphi} \\
        \psi(s; x) & \phantom{:}\approx h(x + s) \label{eq:def-psi} \\
        m(s; x) & := \varphi(s; x) + \psi(s; x). \label{eq:def-m}
      \end{align}
    \end{subequations}

    We make the following assumption on the models~\eqref{eq:def-models}.

    \begin{modelassumption}%
      \label{asm:models}
      For any \(x \in \R^n\), \(\varphi(\cdot; x) \in \mathcal{C}^1\), and satisfies \(\varphi(0; x) = f(x)\) and \(\nabla \varphi(0; x) = \nabla f(x)\).
      For any \(x \in \R^n\) where \(h\) is finite, \(\psi(\cdot; x)\) is proper lsc, and satisfies \(\psi(0; x) = h(x)\) and \(\partial \psi(0; x) = \partial h(x)\).
    \end{modelassumption}

    \begin{proposition}%
      \label{prop:iprox-stationarity}
      Let \(C \subset \R^n\) be nonempty and compact, and let \Cref{asm:models} be satisfied.
      Let~\eqref{eq:nlp} satisfy the constraint qualification~\eqref{eq:cq} at \(x \in C\).
      Assume \(0 \in \argmin{s} m(s; x) + \chi(x + s \mid C)\), and let the latter subproblem satisfy the constraint qualification~\eqref{eq:cq} at \(s = 0\).
      Then \(x\) is first-order stationary for~\eqref{eq:nlp}.
    \end{proposition}

    \begin{proof}
      Let \(\chi_{x, C}(s) := \chi(x + s \mid C)\).
      Because \(C\) is compact, \Cref{asm:models} ensures that \(s \mapsto m(s; x) + \chi(x + s \mid C)\) is lsc, bounded below and that a minimizer exists.
      By assumption,
      \(
      0 \in \nabla f(x) + \partial (\psi + \chi_{x, C})(0)
      \).
      By the constraint qualification for the subproblem and \Cref{asm:models},
      \(
      0 \in \nabla f(x) + \partial h(x) + N_C(x)
      \),
      which are the first-order optimality conditions for~\eqref{eq:nlp}.
    \end{proof}

    In the assumptions of \Cref{prop:iprox-stationarity}, satisfaction of the constraint qualification in the subproblem is ensured if \(\partial^{\infty} \psi(0; x) = \partial^{\infty} h(x)\) provided~\eqref{eq:cq} is satisfied at \(x\) for~\eqref{eq:nlp}.





    \section{Diagonal quasi-Newton methods}%
    \label{sec:quasi-Newton}

    The simplest possible diagonal approximation to \(\nabla^2 f(x_k)\) is \(B_k := \sigma_k I\) for some scalar \(\sigma_k\), where \(I\) denotes the identity.
    In the R2 algorithm \citep[Algorithm~\(6.1\)]{aravkin-baraldi-orban-2022}, \(\sigma_k\) is chosen adaptively at each iteration based on progress.
    A more sophisticated strategy consists in choosing \(\sigma_k\) so as to best approximate \(\nabla^2 f(x_k)\), in a sense to be defined, based on recently observed local information.
    That is the logic behind the spectral projected gradient method \citep{birgin-martinez-raydan-2014}, which is easily generalized as a spectral \emph{proximal} gradient method, and where
    \[
      \sigma_k := s_{k-1}^T y_{k-1} / s_{k-1}^T s_{k-1},
      \quad
      s_{k-1} := x_k - x_{k-1},
      \quad
      y_{k-1} := \nabla f(x_k) - \nabla f(x_{k-1}).
    \] 
    This choice of \(\sigma_k\) arises from the secant equation \(\sigma_k s_{k-1} = y_{k-1}\), which, in general, has no solution, but can be solved in the least-squares sense by minimizing \(\|\sigma s_{k-1} - y_{k-1}\|_2\) in terms of \(\sigma\).

    \citet{gilbert-lemarechal-1989} experiment, among others, with diagonal updates of quasi-Newton approximations.
    However, there may be no diagonal solution \(B_k\) to the secant equation \(B_k s_{k-1} = y_{k-1}\), and one must resort to a different approach to update a diagonal approximation.
    \citet{dennis-wolkowicz-1993} introduce the \emph{weak secant} equation \(s_{k-1}^T B_k s_{k-1} = s_{k-1}^T y_{k-1}\) and accompanying update formulae.
    \citet{nazareth-1995} observes that some of those proposed updates can be constrained to update only the diagonal of \(B_k\) but that the resulting \(B_k\) no longer necessarily satisfies the weak secant equation.

    \citet{zhu-nazareth-wolkowicz-1999} supplement the weak secant equation by additionally requiring that \(B_k\) be diagonal, and call the resulting set of conditions the \emph{quasi-Cauchy} conditions.
    In particular, the spectral gradient approximation is the only solution of the quasi-Cauchy conditions that is a multiple of \(I\).
    They derive updates from variational principles resembling those used in classic quasi-Newton updates.
    The first update sets \(B_k\) to the unique solution of
    \begin{equation}%
      \label{eq:qc-indef-problem}
      \minimize{B} \ \|B - B_{k-1}\|_F \quad \st \ s_{k-1}^T B s_{k-1} = s_{k-1}^T y_{k-1},
    \end{equation}
    which is the counterpart of the Powell-symmetric-Broyden (PSB) variational problem.
    The solution of~\eqref{eq:qc-indef-problem} is shown to be
    \begin{equation}%
      \label{eq:qc-indef-soln}
      B_k = B_{k-1} + \frac{s_{k-1}^T (y_{k-1} - B_{k-1} s_{k-1})}{\trace(S_{k-1}^4)} \, S_{k-1}^2,
      \quad
      S_{k-1} := \diag(s_{k-1}).
    \end{equation}
    Like the PSB update,~\eqref{eq:qc-indef-soln} does not possess the hereditary positive-definiteness property.
    However, capturing negative curvature is a desirable feature in trust-region methods.

    Motivated by linesearch methods, \citet{zhu-nazareth-wolkowicz-1999} devise a second strategy in which they update \(B_k^{1/2}\), and that results in hereditary positive definiteness.
    The main idea is analogous to~\eqref{eq:qc-indef-problem} with the objective replaced with \(\|B^{1/2} - B_{k-1}^{1/2}\|_F\).
    The update depends on the root of a scalar nonlinear equation, and is therefore more costly to perform than~\eqref{eq:qc-indef-soln}.

    \citet{andrei-2019} suggests an alternative variational problem in which \(B_k\) is determined as the solution of
    \begin{equation}%
      \label{eq:qc-andrei-problem}
      \minimize{B} \ \tfrac{1}{2} \|B - B_{k-1}\|_F^2 + \trace(B) \quad \st \ s_{k-1}^T B s_{k-1} = s_{k-1}^T y_{k-1},
    \end{equation}
    where minimizing the trace of \(B\) tends to cluster its eigenvalues.
    The solution of~\eqref{eq:qc-andrei-problem} is shown to be
    \begin{equation}%
      \label{eq:qc-andrei-soln}
      B_k = B_{k-1} + \frac{s_{k-1}^T (y_{k-1} + s_{k-1} - B_{k-1} s_{k-1})}{\trace(S_{k-1}^4)} \, S_{k-1}^2 - I.
    \end{equation}
    Again,~\eqref{eq:qc-andrei-soln} does not preserve positive definiteness.
    Also motivated by linesearch methods, \citet{andrei-2019} employs a thresholding strategy that ensures computation of a descent direction.

    Though \citet{nazareth-1995} appears to have been motivated by derivative-free methods,\footnote{In the sense that the quasi-Cauchy condition only requires the gradient via \(s^T y\).} diagonal quasi-Newton updates are good candidate approximations for use in first-order methods for regularized optimization, and we are not aware of any prior work using them.

    \section{The indefinite proximal gradient method}

    \subsection{Description of the algorithm and convergence properties}
    \label{sec:trdh}

    For \(\nu > 0\), define
    \begin{subequations}%
      \label{eq:def-model-nu}
      \begin{align}
        \varphi_{\textup{cp}}(s; x) & := f(x) + \nabla f(x)^T s,
        \label{eq:def-phi-nu} \\
        m(s; x, \nu) & := \varphi_{\textup{cp}}(s; x) + \tfrac{1}{2} \nu^{-1} \|s\|_2^2 + \psi(s; x).
      \end{align}
    \end{subequations}
    Guided by \Cref{prop:iprox-stationarity}, we compute a first step denoted \(s_{k,1}\) such that
    \begin{equation}%
      \label{eq:ipg-iter-sk1}
      s_{k,1} \in \argmin{s} \ m(s; x_k, \nu_k) + \chi(x_k + s \mid [\ell, \, u] \cap  (x_k + \Delta_k \B)),
    \end{equation}
    for an appropriate value of \(\nu_k > 0\).

    Let
    \begin{equation}%
      \label{eq:def-model-quad}
        m(s; x, D)  := \varphi_{\textup{cp}}(s; x) + \tfrac{1}{2} s^T D s + \psi(s; x),
    \end{equation}
    where \(D\) is a diagonal matrix.
    The indefinite proximal gradient iteration with trust region for~\eqref{eq:nlp} is defined by changing~\eqref{eq:ipg-iter-sk1} to
    \begin{equation}%
      \label{eq:ipg-iter-s-tr}
      s_k \in \argmin{s} \ m(s; x_k, D_k) + \chi(x_k + s \mid [\ell, \, u] \cap  (x_k + \Delta_k \B)),
    \end{equation}
    where \(D_k\) is diagonal, and by updating \(x_k\) and \(\Delta_k\) as is customary in trust-region methods \citep{conn-gould-toint-2000}.
    Among other possible choices, we focus on the case where \(D_k\) results from a diagonal quasi-Newton update.

    We summarize the entire procedure as \Cref{alg:tr-nonsmooth-diag}, which is a special case of \citep[Algorithm~\(3.1\)]{aravkin-baraldi-orban-2022}.

    \begin{algorithm}[ht]
      \caption[caption]{%
        Nonsmooth trust-region algorithm with diagonal Hessian.%
        \label{alg:tr-nonsmooth-diag}
        }
        \begin{algorithmic}[1]
          \State Choose constants
          \[
            0 < \eta_1 \leq \eta_2 < 1,
            \quad
            0 < 1/\gamma_3 \leq \gamma_1 \leq \gamma_2 < 1 < \gamma_3 \leq \gamma_4, \quad \text{and} \quad
            \alpha > 0, \
            \beta \geq 1.
          \]
          \State Choose \(x_0 \in \R^n\) where \(h\) is finite, \(\Delta_0  > 0\), compute \(f(x_0) + h(x_0)\).
          \State Choose \(d_0 \in \R^n\) 
          and set \(D_0 = \diag(d_0)\).
          \For{\(k = 0, 1, \dots\)}
          \State Set \(\nu_k := 1 / (\|d_k\|_{\infty} + \alpha^{-1} \Delta_k^{-1})\). \label{step:vk}
          \State Define \(\varphi(s; x_k)\) and \(\psi(s; x_k)\) according to \Cref{asm:models}.
          \State Define \(m(s; x_k, \nu_k)\) as in~\eqref{eq:def-model-nu} and compute \(s_{k,1}\) as in~\eqref{eq:ipg-iter-sk1}. \label{step:sk1}
          \State Define \(m(s; x_k, D_k)\) as in~\eqref{eq:def-model-quad} and compute a solution \(s_k\) of~\eqref{eq:ipg-iter-s-tr} with \(\Delta_k\) replaced by \(\min (\Delta_k, \, \beta \|s_{k,1}\|)\). \label{step:sk}
          \State Compute the ratio
          \[
            \rho_k :=
            \frac{%
              f(x_k) + h(x_k) - (f(x_k + s_k) + h(x_k + s_k))
              }{
                m(0; x_k, D_k) - m(s_k; x_k, D_k)
              }.
            \]
            \State If \(\rho_k \geq \eta_1\), set \(x_{k+1} = x_k + s_k\).
            Otherwise, set \(x_{k+1} = x_k\).
            \State Choose \(d_{k+1} \in \R^n\) 
            and set \(D_{k+1} = \diag(d_{k+1})\).
            \State Update the trust-region radius according to
            \[
              \Delta_{k+1} \in
              \left\{
                \begin{array}{lll}
                  {[\gamma_3 \Delta_k, \, \gamma_4 \Delta_k]} &
                  \text{ if } \rho_k \geq \eta_2, &
                  \text{(very successful iteration)}
                  \\ {[\gamma_2 \Delta_k, \, \Delta_k]} &
                  \text{ if } \eta_1 \leq \rho_k < \eta_2, &
                  \text{(successful iteration)}
                  \\ {[\gamma_1 \Delta_k, \, \gamma_2 \Delta_k]} &
                  \text{ if } \rho_k < \eta_1 &
                  \text{(unsuccessful iteration).}
                \end{array}
                \right.
              \]
              \EndFor
        \end{algorithmic}
    \end{algorithm}

    In \citep[Algorithm~\(3.1\)]{aravkin-baraldi-orban-2022}, convergence hinges crucially on the fact that the decrease in \(m(\cdot; x_k, D_k)\) achieved by \(s_k\) is at least a fraction of the decrease in \(m(\cdot; x_k, \nu_k)\) achieved by \(s_{k,1}\), the first proximal gradient step with a well-chosen step size \(\nu_k > 0\).
    At iteration \(k\) of \Cref{alg:tr-nonsmooth-diag}, \(\nu_k < 1 / \|d_k\|_{\infty}\).
    Note that the choice \(d_k = 0\) is allowed.

    In \Cref{alg:tr-nonsmooth-diag}, the computation of \(s_{k,1}\) serves two purposes.
    The first is as stopping condition by \Cref{prop:iprox-stationarity}, and the second is to set the trust-region radius in the subproblem for \(s_k\).
    This is at variance with \citep[Algorithm~\(3.1\)]{aravkin-baraldi-orban-2022}, where, in addition, \(s_k\) is computed by continuing the proximal gradient iterations from \(s_{k,1}\).


    In a similar notation to that of \citep{aravkin-baraldi-orban-2022}, let
    \begin{subequations}%
      \label{eq:def-xicp-xi}
      \begin{align}
        \xi_{\textup{cp}}(\Delta_k; x_k, \nu_k) & := f(x_k) + h(x_k) - (\varphi_{\textup{cp}}(s_{k,1}; x_k) + \psi(s_{k,1}; x_k)), \label{eq:def-xicp} \\
        \xi(\Delta_k; x_k, D_k) & := f(x_k) + h(x_k) - m(s_k; x_k, D_k)
      \end{align}
    \end{subequations}
    denote the optimal model decrease for~\eqref{eq:def-model-nu} and~\eqref{eq:def-model-quad}.
    The definition of~\eqref{eq:ipg-iter-sk1} guarantees that
    \begin{equation}
      f(x_k) + h(x_k) = m(0; x_k, \nu_k) \ge m(s_{k,1}; x_k, \nu_k) = \varphi_{\textup{cp}}(s_{k,1}; x_k) + \psi(s_{k,1}; x_k) + \tfrac{1}{2}\nu_k^{-1}\|s_{k,1}\|^2,
    \end{equation}
    which implies
    \begin{equation}%
      \label{eq:suff-decrease}
      \xi_{\textup{cp}}(\Delta_k; x_k, \nu_k) \geq \tfrac{1}{2} \nu_k^{-1} \|s_{k,1}\|^2
      = \tfrac{1}{2} (\|d_k\|_{\infty} + \alpha^{-1}\Delta_k^{-1}) \|s_{k,1}\|^2
      \ge \tfrac{1}{2} \alpha^{-1} \Delta_k^{-1} \|s_{k,1}\|^2.
    \end{equation}
    If \(\xi_{\textup{cp}}(\Delta_k; x_k, \nu_k) = 0\), we obtain \(s_{k,1} = 0\).
    \Cref{prop:iprox-stationarity} then implies that \(x_k\) is first-order stationary provided~\eqref{eq:nlp} satisfies the constraint qualification at \(x_k\) and~\eqref{eq:def-model-nu} satisfies it at \(s = 0\).

    To establish convergence properties of \Cref{alg:tr-nonsmooth-diag}, we require \citeauthor{aravkin-baraldi-orban-2022}'s Step Assumption~\(3.1\), recalled in the following assumption for convenience.
    \begin{stepassumption}%
      \label{asm:step-assumption}
      There exists \(\kappa_{\textup{m}} > 0\) and \(\kappa_{\textup{mdc}} \in (0, 1)\) such that for all \(k\),
      \begin{subequations}
        \begin{align}
          |f(x_k + s_k) + h(x_k + s_k) - (\varphi(s_k; x_k, D_k) + \psi(s_k; x_k))| &\leq \kappa_{\textup{m}} \|s_k\|_2^2, \label{eq:asm-diff-model-predict}\\
          \xi(\Delta_k; x_k, D_k) &\ge \kappa_{\textup{mdc}} \xi_{\textup{cp}}(\Delta_k; x_k; \nu_k) \label{eq:asm-mdc}.
        \end{align}
      \end{subequations}
    \end{stepassumption}

    Among other situations, \eqref{eq:asm-diff-model-predict} is satisfied if \(\{D_k\}\) is bounded, \(\nabla f\) is Lipschitz continuous, and we select \(\psi(s; x_k) = h(x_k + s)\).
    The following proposition gives a sufficient condition for which~\eqref{eq:asm-mdc} is satisfied.
    \begin{proposition}
      \label{prop:mdc}
      If
      \begin{equation}
        \label{eq:dk-ineq}
        \|d_k\|_{\infty} \le (\kappa_{\textup{mdc}}^{-1} - 1) \alpha^{-1} \Delta_k^{-1},
      \end{equation}
      then~\eqref{eq:asm-mdc} is satisfied.
    \end{proposition}
    \begin{proof}
      By definition of \(s_k\),
      \begin{align*}
        m(s_k; x_k, D_k) & \leq m(s_{k,1}; x_k, D_k) \\
        & = \varphi(s_{k,1}; x_k) + \tfrac{1}{2} s_{k,1}^T D_k s_{k,1} + \psi(s_{k,1}; x_k) \\
        & \leq \varphi(s_{k,1}; x_k) + \tfrac{1}{2} \Vert d_k \Vert_{\infty} \Vert s_{k,1} \Vert^2 + \psi(s_{k,1}; x_k) \\
        & = \varphi(s_{k,1}; x_k) + \tfrac{1}{2} (\nu_k^{-1} - \alpha^{-1} \Delta_k^{-1}) \Vert s_{k,1} \Vert^2 + \psi(s_{k,1}; x_k),
      \end{align*}
      which leads to
      \begin{equation*}
        \xi(\Delta_k; x_k, D_k) \geq \xi_{\textup{cp}}(\Delta_k; x_k, \nu_k) - \tfrac{1}{2} (\nu_k^{-1} - \alpha^{-1} \Delta_k^{-1}) \|s_{k,1}\|^2.
      \end{equation*}
      For~\eqref{eq:asm-mdc} to be satisfied, it is sufficient to show that
      \begin{equation*}
        \xi_{\textup{cp}}(\Delta_k; x_k, \nu_k) - \tfrac{1}{2} (\nu_k^{-1} - \alpha^{-1} \Delta_k^{-1}) \|s_{k,1}\|^2 \ge \kappa_{\textup{mdc}} \xi_{\textup{cp}}(\Delta_k; x_k, \nu_k),
      \end{equation*}
      that we rewrite as
      \begin{equation}
        \label{eq:ineq-xi-mdc}
        (1 - \kappa_{\textup{mdc}}) \xi_{\textup{cp}}(\Delta_k; x_k, \nu_k) \geq \tfrac{1}{2} (\nu_k^{-1} - \alpha^{-1} \Delta_k^{-1}) \|s_{k,1}\|^2.
      \end{equation}
      Let \(\|d_k\|_{\infty} \le (\kappa_{\textup{mdc}}^{-1} - 1) \alpha^{-1} \Delta_k^{-1}\), which is equivalent to \(\nu_k = \dfrac{1}{\|d_k\|_{\infty} + \alpha^{-1} \Delta_k^{-1}} \ge \kappa_{\textup{mdc}} \alpha \Delta_k\).
      Since \(\kappa_{\textup{mdc}} \in (0, 1)\), we have
      \begin{equation*}
        (1 - \kappa_{\textup{mdc}}) \tfrac{1}{2} \nu_k^{-1} \ge \tfrac{1}{2}(\nu_k^{-1} - \alpha^{-1} \Delta_k^{-1}).
      \end{equation*}
      By multiplying by \(\|s_{k,1}\|^2\) and recalling that \(\xi_{\textup{cp}}(\Delta_k; x_k, \nu_k) \ge \tfrac{1}{2}\nu_k^{-1}\|s_{k,1}\|^2\), we get
      \begin{equation*}
        (1 - \kappa_{\textup{mdc}}) \xi_{\textup{cp}}(\Delta_k; x_k, \nu_k) \ge (1 - \kappa_{\textup{mdc}}) \tfrac{1}{2} \nu_k^{-1} \|s_{k,1}\|^2 \ge \tfrac{1}{2}(\nu_k^{-1} - \alpha^{-1} \Delta_k^{-1}) \|s_{k,1}\|^2,
      \end{equation*}
      so that~\eqref{eq:ineq-xi-mdc} is satisfied.
    \end{proof}

    We can always choose \(\kappa_{\textup{mdc}}\) small enough to ensure that \(\|d_k\|_{\infty}\) does not have to be too close to zero if \(\Delta_k\) does not grow unbounded.
    Whenever \Cref{alg:tr-nonsmooth-diag} generates infinitely many very successful iterations, we could define a \(\Delta_{\max}\) from which we stop increasing \(\Delta_k\).
    We point out that \Cref{prop:mdc} only gives a sufficient condition for~\eqref{eq:asm-mdc} to be satisfied, but it is not necessary.
    Therefore, for the rest of this paper, we base our analysis on~\eqref{eq:asm-mdc} rather that on~\eqref{eq:dk-ineq}.

    Under \Cref{asm:step-assumption}, we can apply directly \citeauthor{aravkin-baraldi-orban-2022}'s convergence properties.

    \begin{proposition}[{\protect \citealp[Theorem~\(3.4\)]{aravkin-baraldi-orban-2022}}]
      \label{prop:delta-succ}
      Let \Cref{asm:step-assumption} be satisfied and
      \[
        \Delta_{\textup{succ}} := \frac{\kappa_{\textup{mdc}}(1 - \eta_2)}{2 \kappa_{\textup{m}} \alpha \beta^2} > 0.
      \]
      If \(x_k\) is not first-order stationary and \(\Delta_k \leq \Delta_{\textup{succ}}\), iteration \(k\) is very successful and \(\Delta_{k+1} > \Delta_k\).
    \end{proposition}

    Without further assumptions, the following result also holds.

    \begin{proposition}[{\protect \citealp[Theorem~\(3.5\)]{aravkin-baraldi-orban-2022}}]
      Let \Cref{asm:step-assumption} be satisfied and assume \Cref{alg:tr-nonsmooth-diag} generates a finite number of successful iterations.
      Then \(x_k = x_\star\) for all sufficiently large \(k\).
      If~\eqref{eq:nlp} satisfies the constraint qualification at \(x_k\) and~\eqref{eq:def-model-nu} satisfies it at \(s = 0\), \(x_\star\) is first-order critical.
    \end{proposition}

    We continue to follow the analysis of \citep{aravkin-baraldi-orban-2022} and note that \(\Delta_k \geq \Delta_{\min}\) for all \(k \in \N\), where \(\Delta_{\min} := \min(\Delta_0, \, \gamma_1 \Delta_{\textup{succ}})\).
    \citet{aravkin-baraldi-orban-2022} establish that \(\xi_{\textup{cp}}(\cdot ; x_k, \nu_k)\) is increasing, and therefore, \(\xi_{\textup{cp}}(\Delta_k; x_k, \nu_k) \geq \xi_{\textup{cp}}(\Delta_{\min}; x_k, \nu_k)\).
    In addition, they show that \(\nu_k^{-1/2} \xi_{\textup{cp}}(\Delta_k; x_k, \nu_k)^{1/2}\) is an appropriate criticality measure.
    Let \(0 < \epsilon < 1\), and
    \begin{align*}
      I(\epsilon) & := \{ k \in \N \mid \nu_k^{-1/2} \xi_{\textup{cp}}(\Delta_k; x_k, \nu_k)^{1/2} > \epsilon \}, \\ 
      S(\epsilon) & := \{ k \in I(\epsilon) \mid \rho_k \geq \eta_1 \}, \\
      U(\epsilon) & :=  \{ k \in I(\epsilon) \mid \rho_k < \eta_1 \},
    \end{align*}
    be the set of iterations, successful iterations, and unsuccessful iterations until the criticality measure drops below \(\epsilon\), respectively.

    We now derive bounds on \(|I(\epsilon)|\) using the analysis of \citep{aravkin-baraldi-orban-2022} specialized to \Cref{alg:tr-nonsmooth-diag} under the assumption that \(\{D_k\}\) is bounded.
    We make the following assumption.
    \begin{assumption}%
      \label{asm:D-bounded}
      There exists \(d_{\max} > 0\) such that \(\|d_k\|_{\infty} \leq d_{\max}\) for all \(k \in \N\).
    \end{assumption}

    In \Cref{alg:tr-nonsmooth-diag}, it is not difficult to ensure satisfaction of \Cref{asm:D-bounded}.
    For instance, one may prescribe a value \(d_{\max} > 0\) and reset each \(d_k\) componentwise to \(\min(\max(d_k, -d_{\max}), d_{\max})\).
    We point out that \Cref{asm:D-bounded} holds if~\eqref{eq:dk-ineq} in \Cref{prop:mdc} is satisfied for all \(k\), because, in this case, \(\|d_k\|_{\infty} \le d_{\max} := (\kappa_{\textup{mdc}}^{-1} - 1) \alpha^{-1} \Delta_{\min}^{-1}\).

    Under \Cref{asm:D-bounded}, the choice of \(\nu_k\) in \Cref{alg:tr-nonsmooth-diag} guarantees that
    \[
      \nu_k \geq \nu_{\min} := 1 / (d_{\max} + \alpha^{-1} \Delta_{\min}^{-1}) > 0.
    \]

    The following result establishes that \(|S(\epsilon)|\) is \(O(\epsilon^{-2})\).

    \begin{proposition}[{\protect \citealp[Lemma~\(3.6\)]{aravkin-baraldi-orban-2022}}]%
      \label{prop:num-successful}
      Let \Cref{asm:step-assumption} and \Cref{asm:D-bounded} be satisfied.
      Assume that \Cref{alg:tr-nonsmooth-diag} generates infinitely many successful iterations and that there exists \((f + h)_{\textup{low}} \in \R\) such that \((f + h)(x_k) \geq (f + h)_{\textup{low}}\) for all \(k \in \N\).
      Let \(\epsilon \in (0, \, 1)\).
      Then,
      \[
        |S(\epsilon)| \leq \frac{(f + h)(x_0) - (f + h)_{\textup{low}}}{\eta_1 \kappa_{\textup{mdc}} \nu_{\min} \epsilon^2}.
      \]
    \end{proposition}

    The next result establishes that \(|U(\epsilon)|\) is also \(O(\epsilon^{-2})\).

    \begin{proposition}[{\protect \citealp[Lemma~\(3.7\)]{aravkin-baraldi-orban-2022}}]%
      \label{prop:num-unsuccessful}
      Under the assumptions of \Cref{prop:num-successful},
      \[
        |U(\epsilon)| \leq \log_{\gamma_2} (\Delta_{\min} / \Delta_0) + |S(\epsilon)| \, |\log_{\gamma_2}(\gamma_4)|.
      \]
    \end{proposition}

    The direct consequence of \Cref{prop:num-successful,prop:num-unsuccessful} is that \(|I(\epsilon)|\) is also \(O(\epsilon^{-2})\).
    Because \(\epsilon \in (0, \, 1)\) is arbitrary, we conclude that \(\liminf \nu_k^{-1/2} \xi_{\textup{cp}}(\Delta_k; x_k, \nu_k)^{1/2} = 0\).

    \subsection{A variant saving proximal operator computations}
    \label{sec:itrdh}

    We present a variant of \Cref{alg:tr-nonsmooth-diag} in which we set
    \begin{equation}
      \label{eq:vk-itrdh}
      \nu_k := 1 / (\|d_k\|_{\infty} + \alpha^{-1})
    \end{equation}
    in \cref{step:vk}, we do not compute \(s_{k,1}\) in \cref{step:sk1}, and we leave \(\Delta_k\) unchanged in~\cref{step:sk}, thus saving a proximal operator computation at each iteration.
    Under \Cref{asm:D-bounded}, this choice of \(\nu_k\) leads to a new \(\nu_{\min}\)
    \begin{equation}
      \label{eq:vk-min-itrdh}
      \nu_k \geq \nu_{\min} := 1 / (d_{\max} + \alpha^{-1}) > 0.
    \end{equation}
    
    We use the criticality measure
    \begin{equation}
      \nu_k^{-1/2} \xi(\Delta_k; x_k, D_k)^{1/2},
    \end{equation}
    and modify the definition of \(I(\epsilon)\), \(S(\epsilon)\) and \(U(\epsilon)\) accordingly.
    To establish the same convergence properties as those of \Cref{alg:tr-nonsmooth-diag}, we need \Cref{asm:step-assumption} to hold for our algorithmic variant.
    The following assumption summarizes our requirements.
    \begin{stepassumption}%
      \label{asm:step-assumption-itrdh}
      There exists \(\kappa_{\textup{m}} > 0\) and \(\kappa_{\textup{mdc}} \in (0, 1)\) such that for all \(k\),
      \begin{subequations}
        \begin{align}
          |f(x_k + s_k) + h(x_k + s_k) - (\varphi(s_k; x_k, D_k) + \psi(s_k; x_k))| &\leq \kappa_{\textup{m}} \|s_k\|_2^2, \label{eq:asm-diff-model-predict-itrdh}\\
          \xi(\Delta_k; x_k, D_k) &\ge \kappa_{\textup{mdc}} \xi_{\textup{cp}}(\Delta_k; x_k; \nu_k) \label{eq:asm-mdc-itrdh},
        \end{align}
      \end{subequations}
      where \(\nu_k = 1 / (\|d_k\|_{\infty} + \alpha^{-1})\) in \cref{step:vk} and \(s_k\) is computed without changing \(\Delta_k\) in \cref{step:sk} of \Cref{alg:tr-nonsmooth-diag}.
    \end{stepassumption}
    Even though \Cref{asm:step-assumption-itrdh} involves \(\xi_{\textup{cp}}(\Delta_k; x_k, \nu_k)\), we do not need to compute it in practice.
    Conveniently, if \(\{D_k\}\) is bounded, \(\nabla f\) is Lipschitz continuous, and we select \(\psi(s; x_k) = h(x_k + s)\), \eqref{eq:asm-diff-model-predict-itrdh} is still satisfied.
    The following proposition, analogous to \Cref{prop:mdc}, shows that~\eqref{eq:dk-ineq-itrdh} is a sufficient condition for~\eqref{eq:asm-mdc-itrdh} to hold.

    \begin{proposition}
      \label{prop:mdc-itrdh}
      If
      \begin{equation}
        \label{eq:dk-ineq-itrdh}
        \|d_k\|_{\infty} \le (\kappa_{\textup{mdc}}^{-1} - 1) \alpha^{-1},
      \end{equation}
      then~\eqref{eq:asm-mdc-itrdh} is satisfied for \Cref{alg:tr-nonsmooth-diag} with \(\nu_k = 1 / (\|d_k\|_{\infty} + \alpha^{-1})\) in \cref{step:vk}, \(s_{k,1}\) is not computed in \cref{step:sk1}, and \(\Delta_k\) is unchanged in~\cref{step:sk}.
    \end{proposition}
    \begin{proof}
      The proof is identical to that of \Cref{prop:mdc} if we use~\eqref{eq:vk-itrdh}.
    \end{proof}

    Whenever~\eqref{eq:asm-mdc-itrdh} holds, we have
    \begin{equation}
      \label{eq:itrdh-crit}
      \nu_k^{-1/2} \xi(\Delta_k; x_k, \nu_k)^{1/2} \ge \kappa_{\textup{mdc}}^{1/2} \nu_k^{-1/2} \xi_{\textup{cp}}(\Delta_k; x_k, \nu_k)^{1/2} \ge \kappa_{\textup{mdc}}^{1/2} \nu_k^{-1} \|s_{k,1}\|,
    \end{equation}
    which ensures that \(\nu_k^{-1/2} \xi(\Delta_k; x_k, \nu_k)^{1/2} \rightarrow 0\) implies \(\|s_{k,1}\| \rightarrow 0\), and that \(\nu_k^{-1/2} \xi(\Delta_k; x_k, \nu_k)^{1/2}\) is an appropriate criticality measure.
    We emphasize that there is no need to compute \(s_{k,1}\) and \(\xi_{\textup{cp}}(\Delta_k; x_k, \nu_k)\);~\eqref{eq:itrdh-crit} is only used as a theoretical justification for the choice of our criticality measure.

    We now establish convergence properties similar to those of \Cref{sec:trdh}, and follow the analysis of \citet[Chapter~\(2.3\)]{cartis-gould-toint-2022}.
    \begin{proposition}
      \label{prop:delta-up-bnd-itrdh}
      Let \Cref{asm:D-bounded} and \Cref{asm:step-assumption-itrdh} be satisfied.
      Let $\nu_{\min}$ be as defined in~\eqref{eq:vk-min-itrdh}.
      If
      \begin{equation}
        \label{eq:delta-upper-bnd-itrdh}
        \Delta_k \le \sqrt{\frac{\nu_{\min}(1 - \eta_2)}{\kappa_{\textup{m}}}} \nu_k^{-1/2} \xi(\Delta_k; x_k, D_k)^{1/2}
      \end{equation}
      and \(x_k\) is not first-order stationary, iteration \(k\) is very successful and \(\Delta_{k+1} > \Delta_k\).
    \end{proposition}
    \begin{proof}
      Using \Cref{asm:step-assumption-itrdh} and \(\|s_k\| \le \Delta_k\), we have
      \begin{equation*}
        \begin{aligned}
          |\rho_k - 1| &= \left|\frac{f(x_k + s_k) + h(x_k + s_k) - m(s_k; x_k, D_k)}{m(0; x_k, D_k) - m(s_k; x_k, D_k)}\right| \\
          &\le \frac{\kappa_{\textup{m}} \|s_k\|^2}{\xi(\Delta_k; x_k, D_k)} \\
          &\le \frac{\kappa_{\textup{m}} \Delta_k^2}{\xi(\Delta_k; x_k, D_k)}.
        \end{aligned}
      \end{equation*}
      \Cref{asm:D-bounded} and~\eqref{eq:vk-min-itrdh} lead to
      \begin{equation*}
        |\rho_k - 1| \le \frac{\kappa_{\textup{m}} \nu_{\min}^{-1} \Delta_k^2}{\nu_k^{-1} \xi(\Delta_k; x_k, D_k)}.
      \end{equation*}
      Whenever~\eqref{eq:delta-upper-bnd-itrdh} holds, we have \(\rho_k \ge \eta_2\), implying that iteration \(k\) is very successful and \(\Delta_{k+1} \ge \Delta_k\).
    \end{proof}

    The following lemma is inspired by \citet[Theorem~\(6.4.3\)]{conn-gould-toint-2000} and \citet[Lemma~\(2.3.4\)]{cartis-gould-toint-2022}.
    \begin{lemma}
      \label{lem:delta-min-itrdh}
      Let \Cref{asm:step-assumption-itrdh} be satisfied.
      Then, for all \(k \ge 0\),
      \begin{equation}
        \label{eq:delta-min-itrdh}
        \Delta_k > \gamma_1 \sqrt{\frac{\nu_{\min}(1 - \eta_2)}{\kappa_{\textup{m}}}} \min \left(1, \Delta_0 \frac{\sqrt{\kappa_{\textup{m}} \nu_{\min}^{-1}}}{\nu_0^{-1/2} \xi(\Delta_0; x_0, D_0)^{1/2}} \right) \min_{i \in [0, k]} \nu_i^{-1/2} \xi(\Delta_i; x_i, D_i)^{1/2}.
      \end{equation}
    \end{lemma}

    \begin{proof}
      We proceed as in \citet[Lemma~\(2.3.4\)]{cartis-gould-toint-2022}.
      The bound certainly holds for \(k = 0\), because
      \begin{equation*}
        \Delta_0 > \gamma_1 \sqrt{1 - \eta_2} \Delta_0.
      \end{equation*} 
      Now, we proceed by contradiction and assume that \(k \ge 1\) is the first iteration such that~\eqref{eq:delta-min-itrdh} is not verified.
      \begin{equation*}
        \begin{aligned}
          \Delta_k &\le \gamma_1 \sqrt{\frac{\nu_{\min}(1 - \eta_2)}{\kappa_{\textup{m}}}} \min \left(1, \Delta_0 \frac{\sqrt{\kappa_{\textup{m}} \nu_{\min}^{-1}}}{\nu_0^{-1/2}\xi(\Delta_0; x_0, D_0)^{1/2}} \right) \min_{i \in [0, k]} \nu_i^{-1/2} \xi(\Delta_i; x_i, D_i)^{1/2} \\
          &\le \gamma_1 \sqrt{\frac{\nu_{\min}(1 - \eta_2)}{\kappa_{\textup{m}}}} \nu_{k-1}^{-1/2} \xi(\Delta_{k-1}; x_{k-1}, D_{k-1})^{1/2}.
        \end{aligned}
      \end{equation*}
      We have \(\gamma_1 \Delta_{k-1} \le \Delta_k\) because of the updating rules of \(\Delta_k\) in \Cref{alg:tr-nonsmooth-diag}, which implies that 
      \begin{equation*}
        \Delta_{k-1} \le \sqrt{\frac{\nu_{\min}(1 - \eta_2)}{\kappa_{\textup{m}}}} \nu_{k-1}^{-1/2} \xi(\Delta_{k-1}; x_{k-1}, D_{k-1})^{1/2}.
      \end{equation*}
      Using \Cref{prop:delta-up-bnd-itrdh}, this results in iteration \(k-1\) being very successful and \(\Delta_k > \Delta_{k-1}\).
      Therefore,
      \begin{equation*}
        \begin{aligned}
          \Delta_{k-1} < \Delta_k &\le \gamma_1 \sqrt{\frac{\nu_{\min}(1 - \eta_2)}{\kappa_{\textup{m}}}} \min \left(1, \Delta_0 \frac{\sqrt{\kappa_{\textup{m}} \nu_{\min}^{-1}}}{\nu_0^{-1/2} \xi(\Delta_0; x_0, D_0)^{1/2}} \right) \min_{i \in [0, k]} \nu_i^{-1/2} \xi(\Delta_i; x_i, D_i)^{1/2} \\
          &\le \gamma_1 \sqrt{\frac{\nu_{\min}(1 - \eta_2)}{\kappa_{\textup{m}}}} \min \left(1, \Delta_0 \frac{\sqrt{\kappa_{\textup{m}} \nu_{\min}^{-1}}}{\nu_0^{-1/2} \xi(\Delta_0; x_0, D_0)^{1/2}} \right) \min_{i \in [0, k-1]} \nu_i^{-1/2} \xi(\Delta_i; x_i, D_i)^{1/2},
        \end{aligned}
      \end{equation*}
      which is a contradiction with iteration \(k\) being the first to violate~\eqref{eq:delta-min-itrdh}.
    \end{proof}

    Now, we show analogous properties to \Cref{prop:num-successful} and \Cref{prop:num-unsuccessful} for our variant.

    \begin{lemma}
      \label{lem:num-successful-itrdh}
      Let \Cref{asm:step-assumption-itrdh} and \Cref{asm:D-bounded} be satisfied.
      Assume that \Cref{alg:tr-nonsmooth-diag} with \(\nu_k\) as in~\eqref{eq:vk-itrdh} in \cref{step:vk}, without \cref{step:sk1}, and with \(\Delta_k\) unchanged in \cref{step:sk} generates infinitely many successful iterations and that there exists \((f + h)_{\textup{low}} \in \R\) such that \((f + h)(x_k) \geq (f + h)_{\textup{low}}\) for all \(k \in \N\).
      Let \(\epsilon \in (0, \, 1)\).
      Then,
      \[
        |S(\epsilon)| \leq \frac{(f + h)(x_0) - (f + h)_{\textup{low}}}{\eta_1 \nu_{\min} \epsilon^2}.
      \]
    \end{lemma}

    \begin{proof}
      For \(k \in S(\epsilon)\), we have
      \begin{equation*}
        \begin{aligned}
          f(x_k) + h(x_k) - f(x_k + s_k) - h(x_k + s_k) &\ge \eta_1 (m(0; x_k, D_k) - m(s_k; x_k, D_k)) \\
          & = \eta_1 \xi(\Delta_k; x_k, D_k) \\
          & \ge \eta_1 \nu_k \epsilon^2\\
          & \ge \eta_1 \nu_{\min} \epsilon^2.
        \end{aligned}
      \end{equation*}
      As \((f + h)(x_k) \ge (f + h)_{\textup{low}}\), we sum the above inequalities for \(k \in S(\epsilon)\) and get
      \begin{equation*}
        (f+h)(x_0) - (f+h)_{\textup{low}} \ge \sum_{k\in S(\epsilon)} (f + h)(x_k) - (f + h)(x_{k+1}) \ge |S(\epsilon)|\eta_1 \epsilon^2 \nu_{\min}.
      \tag*{\qed}
      \end{equation*}
      \renewcommand{\qedsymbol}{}
    \end{proof}

    \begin{lemma}
      \label{lem:num-unsuccessful-itrdh}
      Under the assumptions of \Cref{lem:num-successful-itrdh},
      \[
        |U(\epsilon)| = O(\epsilon^{-2}).
      \]
    \end{lemma}

    \begin{proof}
      Let \(k < k(\epsilon)\).
      Because
      \begin{equation*}
        \nu_k^{-1/2} \xi(\Delta_k; x_k, D_k)^{1/2} > \epsilon,
      \end{equation*}
      and 
      \begin{equation*}
        \nu_{k(\epsilon)}^{-1/2} \xi(\Delta_{k(\epsilon)}; x_{k(\epsilon)}, D_{k(\epsilon)})^{1/2} \leq \epsilon,
      \end{equation*}
      we have with \Cref{lem:delta-min-itrdh}
      \begin{equation*}
        \begin{aligned}
          \Delta_{k(\epsilon) - 1} &> \gamma_1 \sqrt{\frac{\nu_{\min}(1 - \eta_2)}{\kappa_{\textup{m}}}} \min \left(1, \Delta_0 \frac{\sqrt{\kappa_{\textup{m}} \nu_{\min}^{-1}}}{\nu_0^{-1/2} \xi(\Delta_0; x_0, D_0)^{1/2}} \right) \min_{i \in [0, k]} \nu_i^{-1} \xi(\Delta_i; x_i, D_i) \\
          & > \gamma_1 \sqrt{\frac{\nu_{\min}(1 - \eta_2)}{\kappa_{\textup{m}}}} \min \left(1, \Delta_0 \frac{\sqrt{\kappa_{\textup{m}} \nu_{\min}^{-1}}}{\nu_0^{-1/2} \xi(\Delta_0; x_0, D_0)^{1/2}} \right) \epsilon.
        \end{aligned}
      \end{equation*}
      For each successful iteration, \(\Delta_{k+1} \le \gamma_4 \Delta_k\), and for each unsuccessful iteration, \(\Delta_{k+1} \le \gamma_2 \Delta_k\), which implies
      \begin{equation*}
        \Delta_{k(\epsilon) - 1} \le \Delta_0 \gamma_2^{|U(\epsilon)|} \gamma_4^{|S(\epsilon)|}.
      \end{equation*}
      By taking the natural logarithm of the above inequality, we have
      \begin{equation*}
        \begin{aligned}
          |U(\epsilon)|\log(\gamma_2) +|S(\epsilon)| \log(\gamma_4) &\ge \log (\Delta_{k(\epsilon) - 1} / \Delta_0) \\
          & \ge \log \left(\gamma_1 \sqrt{\frac{\nu_{\min}(1 - \eta_2)}{\kappa_{\textup{m}}}} \min \left(\Delta_0^{-1}, \frac{\sqrt{\kappa_{\textup{m}} \nu_{\min}^{-1}}}{\nu_0^{-1/2} \xi(\Delta_0; x_0, D_0)^{1/2}} \right) \epsilon \right), \\
        \end{aligned}
      \end{equation*}
      so that, using the previous inequalities and \Cref{lem:num-successful-itrdh}
      \begin{subequations}
        \begin{align*}
          |U(\epsilon)| &\le \log_{\gamma_2} \left(\gamma_1 \sqrt{\frac{\nu_{\min}(1 - \eta_2)}{\kappa_{\textup{m}}}} \min \left(\Delta_0^{-1}, \frac{\sqrt{\kappa_{\textup{m}} \nu_{\min}^{-1}}}{\nu_0^{-1/2} \xi(\Delta_0; x_0, D_0)^{1/2}} \right) \epsilon \right) + |S(\epsilon)| |\log_{\gamma_2}(\gamma_4)| \\
          &\le O(|\log_{\gamma_2} (\epsilon)|) + O(\epsilon^{-2}) \\
          &= O(\epsilon^{-2}).
          \tag*{\qed} 
        \end{align*}
      \end{subequations}
      \renewcommand{\qedsymbol}{}
    \end{proof}

    \Cref{lem:num-successful-itrdh} and \Cref{lem:num-unsuccessful-itrdh} also indicate that \(I(\epsilon)\) is \(O(\epsilon^{-2})\), and \(\liminf \nu_k^{-1/2} \xi(\Delta_k; x_k, D_k)^{1/2} = 0\).
    With~\eqref{eq:asm-mdc-itrdh} in \Cref{asm:step-assumption-itrdh}, we also have \(\liminf \nu_k^{-1/2} \xi_{\textup{cp}}(\Delta_k; x_k, \nu_k)^{1/2} = 0\).

    Finally, we emphasize the changes made in this section to \Cref{alg:tr-nonsmooth-diag}:
    \begin{itemize}
      \item in \cref{step:vk}, \(\nu_k = 1 / (\|d_k\|_{\infty} + \alpha^{-1} \Delta_k^{-1})\) became \(\nu_k = 1 / (\|d_k\|_{\infty} + \alpha^{-1})\),
      \item \cref{step:sk1} of \Cref{alg:tr-nonsmooth-diag} was removed,
      \item \(\Delta_k\) was left unchanged in \cref{step:sk} of \Cref{alg:tr-nonsmooth-diag},
      \item we used \(\nu_k^{-1/2}\xi(\Delta_k; x_k, D_k)^{1/2}\) as a criticality measure, instead of \(\nu_k^{-1/2}\xi_{\textup{cp}}(\Delta_k; x_k, \nu_k)\), however, when~\eqref{eq:asm-diff-model-predict-itrdh} holds,~\eqref{eq:itrdh-crit} indicates that if the new criticality measure is small, then the criticality measure of \Cref{sec:trdh} is also small,
      \item in \Cref{prop:delta-up-bnd-itrdh}, there is no \(\Delta_{\min} > 0\) independent of \(k\) as in \Cref{prop:delta-succ},
      \item under the assumptions of \Cref{lem:num-successful-itrdh}, which are similar to those of \Cref{prop:num-successful}, the complexity bound is still in \(O(\epsilon^{-2})\), and \(\liminf \nu_k^{-1/2} \xi(\Delta_k; x_k, D_k)^{1/2} = \liminf \nu_k^{-1/2} \xi_{\textup{cp}}(\Delta_k; x_k, \nu_k)^{1/2} = 0\).
    \end{itemize}

    \section{Implementation and numerical experiments}%
      \label{sec:numerical}

      Our Julia implementation of \Cref{alg:tr-nonsmooth-diag} is available from our RegularizedOptimization package \citep{baraldi-orban-regularized-optimization-2022} under the name TRDH\@.
      The latter can be used to solve~\eqref{eq:nlp} directly, or as subproblem solver in TR \citep[Algorithm~\(3.1\)]{aravkin-baraldi-orban-2022}, instead of R2 \citep[Algorithm~\(6.1\)]{aravkin-baraldi-orban-2022}.
      Below, we use the notation TR-R2 and TR-TRDH to denote the application of TR to solve~\eqref{eq:nlp} with R2 or TRDH as subproblem solver, respectively.

      In our experiments, TR uses either an LSR1 or an LBFGS quasi-Newton Hessian approximation with memory \(5\), as implemented in the LinearOperators package \citep{orban-siqueira-linearoperators-2020}.
      The same package implements diagonal quasi-Newton operators in TRDH using the spectral update, the PSB update~\eqref{eq:qc-indef-soln}, and the update of \citet{andrei-2019}~\eqref{eq:qc-andrei-soln} with the modification that we scaled the weak secant equations as
      \begin{equation}
        \label{eq:scaled-secant}
        \tilde s_{k-1}^T B \tilde s_{k-1} = \tilde s_{k-1}^T \tilde y_{k-1}
      \end{equation}
      in~\eqref{eq:qc-indef-problem} and~\eqref{eq:qc-andrei-problem}, where \(\tilde s_{k-1} := s_{k-1} / \|s_{k-1}\|_2\), and \(\tilde y_{k-1} := y_{k-1} / \|s_{k-1}\|_2\), in order to alleviate numerical issues as \(s_{k-1}\) approaches zero.

      The indefinite proximal operators of \Cref{ex:iprox-l0} and \Cref{ex:iprox-l1} are implemented as part of the ShiftedProximalOperators package \citep{baraldi-orban-shifted-proximal-operators-2022}.

      When using TRDH as the main solver, we initialize \(D_0 := \nu_0^{-1} I\) for \(\nu_0 > 0\) given below.
      When using TR-TRDH with a spectral diagonal quasi-Newton approximation, denoted TR-TRDH-Spec, we set the initial diagonal Hessian approximation in TRDH at iteration \(k\) of TR to \(D_{k, 0} := \nu_k^{-1} I\) (as we would initialize R2 in TR-R2).
      When using TR-TRDH with the PSB or the Andrei quasi-Newton approximations, denoted TR-TRDH-PSB and TR-TRDH-Andrei, respectively, we set \(D_{k, 0} := \diag (B_k)\), where \(B_k\) is the quasi-Newton Hessian approximation at iteration \(k\) of TR\@.

      We set \(\psi (s; x_k) := h(x_k + s)\).
      We initialize \(\nu_0 = 1\) for R2 and TRDH used by themselves.
      The stopping criteria that we used for TR, TRDH and R2 (as subproblem solvers or main solvers)  are based on \(\xi_{\textup{cp}}(\Delta_k; x_k, \nu_k)^{1/2}\).
      We set \(\Delta_0 = 1\) for TR and for TRDH used as main solver.
      For TR-TRDH, at iteration \(k\) of TR, the initial value of the TRDH trust-region radius is \(\Delta_{k, 0} = \min (\Delta_k, \beta \|s_{k,1}\|)\) / 10, where \(\Delta_k\) is the TR trust-region radius at iteration \(k\), and \(s_{k, 1}\) is the first step of the \(k\)-th TR subproblem.
      In other words, the initial TRDH trust-region radius to solve the \(k\)-th TR subproblem is a tenth of the trust-region radius of this \(k\)-th subproblem.

      For all solvers, the outer iterations terminate as soon as
      \begin{equation}
        \label{eq:stop-crit}
        \nu_k^{-1/2} \xi_{\textup{cp}}(\Delta_k; x_k, \nu_k)^{1/2} < \epsilon_a + \epsilon_r \nu_0^{-1/2} \xi_{\textup{cp}}(\Delta_0; x_0, \nu_0)^{1/2},
      \end{equation}
      where \(\epsilon_a > 0\) and \(\epsilon_r > 0\) are an absolute and a relative tolerance.
      A round of inner iterations in TR terminates as soon as the stationarity measure of the inner solver satisfies~\eqref{eq:stop-crit}, with \(\tilde \epsilon_a = 10^{-5}\) for the first TR iteration, otherwise \(\tilde \epsilon_a = \max (\epsilon_{a, i}, \min (10^{-2}, \nu_k^{-1/2}\xi_{\textup{cp}}(\Delta_k; x_k, \nu_k)^{1/2})) \),
      and \(\tilde \epsilon_r = \epsilon_{r, i}\), where \(\epsilon_{a, i} > 0 \) and \(\epsilon_{r, i} > 0 \) are some absolute and relative inner tolerances, and \(s_{k,1}\) is the first iterate of the solution of the trust-region subproblem.
      In the experiments below, we use \(\epsilon_{a, i} = 10^{-3}\) and \(\epsilon_{r, i} = 10^{-6}\), except in \Cref{sec:bpdn}, where we use \(\epsilon_{a, i} = 10^{-5}\).

      Additionally, we test the variant presented in \Cref{sec:itrdh}, which is denoted ``iTRDH'' (\emph{indefinite trust-region with diagonal Hessian approximations}) in our results.
      When in use, the stopping criterion is based on \(\nu_k^{-1/2}\xi(\Delta_k; x_k, D_k)^{1/2}\), instead of \(\nu_k^{-1/2}\xi_{\textup{cp}}(\Delta_k; x_k, \nu_k)^{1/2}\). 

      In our results, we report
      \begin{itemize}
        \item the final \(f(x)\);
        \item the final \(h(x) / \lambda\);
        \item the final stationarity measure \(\sqrt{\xi / \nu}\);
        \item \(\|x - x_\star\|_2\), where \(x_\star\) is the exact solution, if it is available;
        \item the number of smooth objective evaluations \(\#~f\);
        \item the number of gradient evaluations \(\#~\nabla f\);
        \item the number of proximal operator evaluations \(\#~\textup{prox}\);
        \item the elapsed time \(t\) in seconds.  
      \end{itemize}
      Because our implementations are not yet perfectly optimized in terms of memory allocations, we neglect the elapsed time in our interpretations, and only report it in the tables as an indicator for the reader.

      In our test cases, available from the RegularizedProblems package \citep{baraldi-orban-regularized-problems-2022}, the computational cost of evaluating the gradient is significantly higher than a proximal evaluation or an objective evaluation.
      In all cases, except in case of failure, all solvers find similar final solutions, and we only show one for illustration.

      \subsection{Basis pursuit denoise (BPDN)}%
        \label{sec:bpdn}

        Our first test case is the basis pursuit denoise (BPDN) problem \citep{tibshirani-1996, donoho-2006}.
        The stopping tolerances \(\epsilon_a\) and \(\epsilon_r\) are set to \(10^{-5}\).
        In this subsection only, when using TR, we set \(\epsilon_{a, i} = 10^{-5}\) in order to have accurate subproblem solves, which leads to performing fewer gradient evaluations without sacrificing too many proximal evaluations.
        Let \(m = 200\), \(n = 512\), \(b = A x_\star + \epsilon\), where \(\epsilon \sim \mathcal{N}(0, 0.01)\), \(A \in \R^{m \times n}\) has orthonormal rows, and \(x_\star\) is a vector of zeros, except for \(10\) of its components that are set to \(\pm 1\).
        We solve
        \begin{equation}
          \label{eq:bpdn}
          \minimize{x} \tfrac{1}{2} \|A x - b\|_2^2 + h(x),
        \end{equation}
        where \(h(x) = \lambda \|x\|_0\).
        As in \citep{aravkin-baraldi-orban-2022}, we use \(\lambda = 0.1 \|A^T b\|_{\infty}\).

        \begin{table}[ht]
  \centering
  \small
  \caption{%
  \label{tbl:bpdn}
  BPDN~\eqref{eq:bpdn} statistics with $h = \lambda \|\cdot\|_0$.
  All variants of TR use an LSR1 Hessian approximation, and are given a maximum of \(100\) inner iterations.
  The optimal objective value is $f(x_\star) = 9.90e-03$.}
  \begin{tabular}{rrrrrrrrr}
    \hline\hline
    solver & $f(x)$ & $h(x)/\lambda$ & $\sqrt{\xi / \nu}$ & $\|x-x_\star\|_2$ & $\#f$ & $\#\nabla f$ & $\#\prox{}$ & $t$ ($s$) \\\hline
    R2 & \(9.44\)e\(-03\) & \(10\) & \(4.4\)e\(-03\) & \(4.7\)e\(-02\) & \(30\) & \(31\) & \(30\) & \(8.0\)e\(-03\) \\      
    TRDH-Spec & \(9.44\)e\(-03\) & \(10\) & \(4.0\)e\(-03\) & \(4.7\)e\(-02\) & \(9\) & \(9\) & \(17\) & \(1.7\)e\(-02\) \\ 
    iTRDH-Spec & \(9.44\)e\(-03\) & \(10\) & \(3.4\)e\(-03\) & \(4.7\)e\(-02\) & \(10\) & \(9\) & \(9\) & \(4.0\)e\(-03\) \\
    TRDH-PSB & \(9.44\)e\(-03\) & \(10\) & \(3.7\)e\(-03\) & \(4.7\)e\(-02\) & \(18\) & \(15\) & \(35\) & \(1.1\)e\(-02\) \\
    iTRDH-PSB & \(9.44\)e\(-03\) & \(10\) & \(2.6\)e\(-03\) & \(4.7\)e\(-02\) & \(20\) & \(16\) & \(19\) & \(5.0\)e\(-03\) \\
    TRDH-Andrei & \(9.44\)e\(-03\) & \(10\) & \(4.5\)e\(-03\) & \(4.7\)e\(-02\) & \(48\) & \(29\) & \(95\) & \(1.0\)e\(-02\) \\
    iTRDH-Andrei & \(9.44\)e\(-03\) & \(10\) & \(3.7\)e\(-03\) & \(4.7\)e\(-02\) & \(59\) & \(35\) & \(58\) & \(3.1\)e\(-02\) \\
    TR-R2 & \(9.44\)e\(-03\) & \(10\) & \(1.7\)e\(-05\) & \(4.7\)e\(-02\) & \(23\) & \(23\) & \(40\) & \(1.8\)e\(-02\) \\   
    TR-TRDH-PSB & \(9.44\)e\(-03\) & \(10\) & \(1.6\)e\(-05\) & \(4.7\)e\(-02\) & \(21\) & \(21\) & \(58\) & \(6.9\)e\(-02\) \\
    TR-iTRDH-PSB & \(9.44\)e\(-03\) & \(10\) & \(1.6\)e\(-05\) & \(4.7\)e\(-02\) & \(21\) & \(21\) & \(39\) & \(2.5\)e\(-02\) \\
    TR-TRDH-Andrei & \(9.44\)e\(-03\) & \(10\) & \(1.6\)e\(-05\) & \(4.7\)e\(-02\) & \(20\) & \(20\) & \(127\) & \(2.5\)e\(-02\) \\
    TR-iTRDH-Andrei & \(9.44\)e\(-03\) & \(10\) & \(1.4\)e\(-05\) & \(4.7\)e\(-02\) & \(20\) & \(20\) & \(100\) & \(3.0\)e\(-02\) \\
    TR-TRDH-Spec & \(9.44\)e\(-03\) & \(10\) & \(1.9\)e\(-05\) & \(4.7\)e\(-02\) & \(21\) & \(21\) & \(46\) & \(1.8\)e\(-02\) \\
    TR-iTRDH-Spec & \(9.44\)e\(-03\) & \(10\) & \(1.9\)e\(-05\) & \(4.7\)e\(-02\) & \(21\) & \(21\) & \(33\) & \(1.8\)e\(-02\) \\\hline\hline
  \end{tabular}
\end{table}

        \begin{figure}[ht]%
          \centering
          \includetikzgraphics[width = 0.49\linewidth]{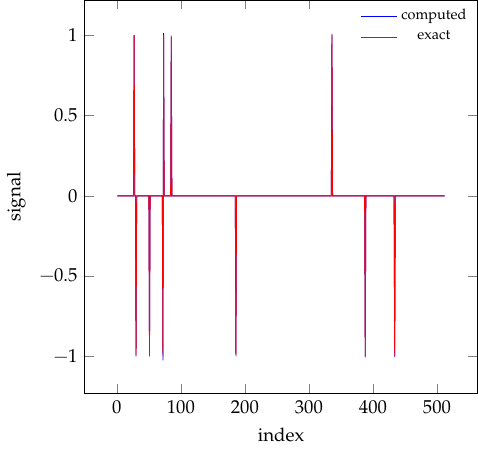}
          \hfill
          \includetikzgraphics[width = 0.49\linewidth]{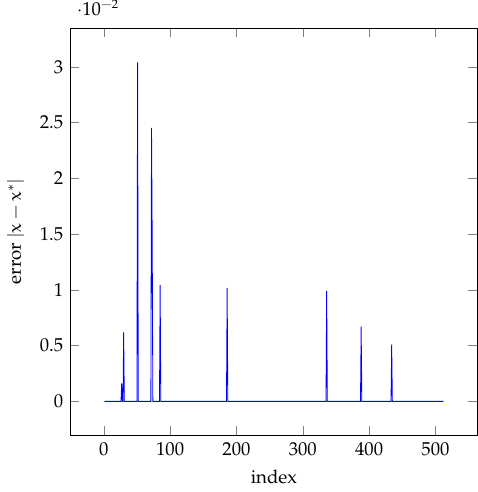}
          \caption{%
            \label{fig:bpdn-sol}
            Solution of~\eqref{eq:bpdn} (left), and error (right) with TRDH-Spec.
          }
        \end{figure}

        \Cref{fig:bpdn-sol} shows the solution of~\eqref{eq:bpdn} with TRDH-Spec.
        \Cref{tbl:bpdn} reports the statistics for the different solvers and shows that TRDH-Spec and TRDH-PSB perform fewer objective and gradient evaluations than R2.
        TRDH-Andrei performs worse that R2 on this problem.
        As expected, the ``iTRDH'' variants require fewer proximal operator evaluations.
        They result in similar numbers of objective and gradient evaluations in the cases of TRDH-Spec and TRDH-PSB, but require more evaluations in the case of TRDH-Andrei.

        All TR-TRDH and TR-iTRDH solvers perform fewer objective and gradient evaluations than TR-R2.

        We now solve the constrained variant
        \begin{equation}%
          \label{eq:bpdn-cstr}
          \minimize{x} \tfrac{1}{2} \|A x - b\|_2^2 + h(x) \quad \st \ x \ge 0,
        \end{equation}
        where each element in \(x_\star\) is either \(0\) or \(1\).

        \begin{table}[ht]
  \centering
  \small
  \caption{%
  \label{tbl:bpdn-cstr}
  Constrained BPDN~\eqref{eq:bpdn-cstr} statistics with $h = \lambda \|\cdot\|_0$.
  All variants of TR use an LSR1 Hessian approximation, and are given a maximum of \(100\) inner iterations.
  The optimal objective value is $f(x_\star) = 8.98e-03$.}
  \begin{tabular}{rrrrrrrrr}
    \hline\hline
    solver & $f(x)$ & $h(x)/\lambda$ & $\sqrt{\xi / \nu}$ & $\|x-x_\star\|_2$ & $\#f$ & $\#\nabla f$ & $\#\prox{}$ & $t$ ($s$) \\\hline
    R2 & \(8.70\)e\(-03\) & \(10\) & \(3.9\)e\(-03\) & \(3.6\)e\(-02\) & \(27\) & \(28\) & \(27\) & \(1.9\)e\(-02\) \\      
    TRDH-Spec & \(8.70\)e\(-03\) & \(10\) & \(3.6\)e\(-03\) & \(3.6\)e\(-02\) & \(8\) & \(8\) & \(15\) & \(4.0\)e\(-03\) \\ 
    iTRDH-Spec & \(8.70\)e\(-03\) & \(10\) & \(3.0\)e\(-03\) & \(3.6\)e\(-02\) & \(9\) & \(8\) & \(8\) & \(2.0\)e\(-03\) \\
    TRDH-PSB & \(8.70\)e\(-03\) & \(10\) & \(3.6\)e\(-03\) & \(3.6\)e\(-02\) & \(11\) & \(11\) & \(21\) & \(3.0\)e\(-03\) \\
    iTRDH-PSB & \(8.70\)e\(-03\) & \(10\) & \(3.7\)e\(-03\) & \(3.6\)e\(-02\) & \(12\) & \(11\) & \(11\) & \(3.0\)e\(-03\) \\
    TRDH-Andrei & \(8.70\)e\(-03\) & \(10\) & \(4.3\)e\(-03\) & \(3.6\)e\(-02\) & \(48\) & \(29\) & \(95\) & \(1.6\)e\(-02\) \\
    iTRDH-Andrei & \(8.70\)e\(-03\) & \(10\) & \(4.0\)e\(-03\) & \(3.6\)e\(-02\) & \(59\) & \(35\) & \(58\) & \(2.0\)e\(-02\) \\
    TR-R2 & \(8.70\)e\(-03\) & \(10\) & \(1.9\)e\(-05\) & \(3.6\)e\(-02\) & \(20\) & \(20\) & \(32\) & \(1.2\)e\(-02\) \\   
    TR-TRDH-PSB & \(8.70\)e\(-03\) & \(10\) & \(2.1\)e\(-05\) & \(3.6\)e\(-02\) & \(19\) & \(19\) & \(50\) & \(1.8\)e\(-02\) \\
    TR-iTRDH-PSB & \(8.70\)e\(-03\) & \(10\) & \(2.1\)e\(-05\) & \(3.6\)e\(-02\) & \(19\) & \(19\) & \(34\) & \(2.5\)e\(-02\) \\
    TR-TRDH-Andrei & \(8.70\)e\(-03\) & \(10\) & \(1.9\)e\(-05\) & \(3.6\)e\(-02\) & \(20\) & \(20\) & \(63\) & \(1.8\)e\(-02\) \\
    TR-iTRDH-Andrei & \(8.70\)e\(-03\) & \(10\) & \(1.5\)e\(-05\) & \(3.6\)e\(-02\) & \(20\) & \(20\) & \(51\) & \(7.0\)e\(-03\) \\
    TR-TRDH-Spec & \(8.70\)e\(-03\) & \(10\) & \(2.0\)e\(-05\) & \(3.6\)e\(-02\) & \(19\) & \(19\) & \(28\) & \(1.0\)e\(-02\) \\
    TR-iTRDH-Spec & \(8.70\)e\(-03\) & \(10\) & \(2.0\)e\(-05\) & \(3.6\)e\(-02\) & \(19\) & \(19\) & \(23\) & \(1.2\)e\(-02\) \\\hline\hline
  \end{tabular}
\end{table}

        \begin{figure}[ht]%
          \centering
          \includetikzgraphics[width = 0.49\linewidth]{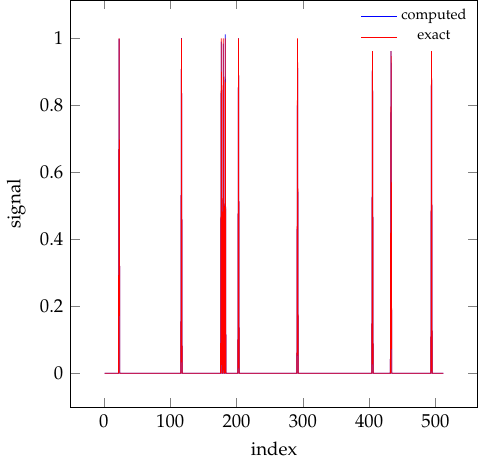}
          \hfill
          \includetikzgraphics[width = 0.49\linewidth]{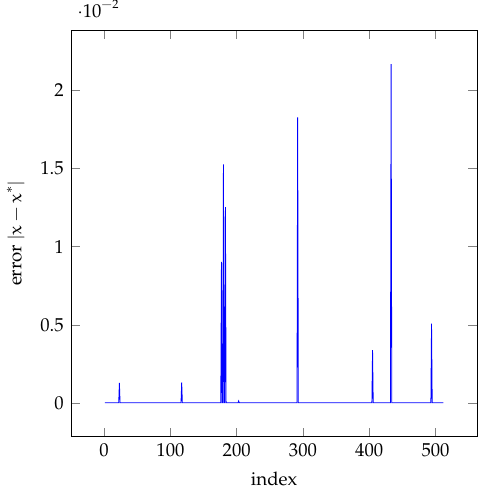}
          \caption{%
            \label{fig:bpdn-cstr-sol}
            Solution of~\eqref{eq:bpdn-cstr} (left), and error (right) with TR-TRDH-Spec.
          }
        \end{figure}

        \Cref{fig:bpdn-cstr-sol} shows the solution of~\eqref{eq:bpdn-cstr} with TRDH-Spec, and \Cref{tbl:bpdn-cstr} reports the statistics for the different solvers.
        We observe that TRDH-Spec, iTRDH-Spec, TRDH-PSB and iTRDH-PSB used as main solvers outperform R2 in terms of number of objective and gradient evaluations.
        TR-TRDH and TR-iTRDH perform similarly to TR-R2.

      \subsection{Sparse nonnegative matrix factorization (NNMF)}

        Our next test case is a variant of the NNMF problem of \citet{kim-park-2008}.
        Let \(A \in \R^{m \times n}\) have nonnegative entries, where each row represents a feature and each column represents an observation.
        We wish to factorize \(A \approx WH\) by separating \(A\) into \(k < \min(m, n)\) clusters, where \(W \in \R^{m \times k}\), \(H \in \R^{k \times n}\) both have nonnegative entries and \(H\) is sparse.
        The problem is stated as
        \begin{equation}%
          \label{eq:nnmf}
          \minimize{W, H} \tfrac{1}{2} \|A - W H\|_F^2 + h(H) \quad \st \ W, H \ge 0,
        \end{equation}
        where \(h(H) = \lambda \|\vec{H}\|_0\) and \(\vec{}\) stacks the columns of a matrix to form a vector.
        In our experiments, each observation is generated using a mixture of Gaussians.
        Negative elements in the matrix \(A\) generated are reset to zero.

        \begin{table}[ht]
  \centering
  \small
  \caption{%
  \label{tbl:nnmf}
  NNMF~\eqref{eq:nnmf} statistics with $h = \lambda \|\cdot\|_0$.
  All variants of TR use an LSR1 Hessian approximation, and are given a maximum of \(100\) inner iterations.}
  \begin{tabular}{rrrrrrrr}
    \hline\hline
    solver & $f(x)$ & $h(x) / \lambda$ & $\sqrt{\xi / \nu}$ & $\#f$ & $\#\nabla f$ & $\#\prox{}$ & $t$ ($s$) \\\hline
    R2 & \(2.84\)e\(+03\) & \(0\) & \(2.2\)e\(-04\) & \(2\) & \(2\) & \(2\) & \(0.0\)e\(+00\) \\
    TRDH-Spec & \(1.25\)e\(+02\) & \(50\) & \(7.7\)e\(-02\) & \(47\) & \(29\) & \(93\) & \(5.0\)e\(-03\) \\
    iTRDH-Spec & \(1.25\)e\(+02\) & \(50\) & \(6.5\)e\(-02\) & \(48\) & \(29\) & \(47\) & \(4.0\)e\(-03\) \\
    TRDH-PSB & \(1.70\)e\(+02\) & \(62\) & \(5.7\)e\(+00\) & \(501\) & \(352\) & \(1000\) & \(5.6\)e\(-02\) \\
    iTRDH-PSB & \(1.70\)e\(+02\) & \(62\) & \(8.4\)e\(+00\) & \(501\) & \(352\) & \(500\) & \(5.5\)e\(-02\) \\
    TRDH-Andrei & \(2.73\)e\(+02\) & \(54\) & \(5.9\)e\(+00\) & \(501\) & \(283\) & \(1000\) & \(5.9\)e\(-02\) \\
    iTRDH-Andrei & \(2.73\)e\(+02\) & \(54\) & \(7.0\)e\(+00\) & \(501\) & \(283\) & \(500\) & \(5.6\)e\(-02\) \\
    TR-R2 & \(1.25\)e\(+02\) & \(50\) & \(5.0\)e\(-03\) & \(212\) & \(115\) & \(4983\) & \(2.4\)e\(-01\) \\
    TR-TRDH-PSB & \(1.25\)e\(+02\) & \(50\) & \(7.2\)e\(-03\) & \(197\) & \(104\) & \(9826\) & \(7.3\)e\(-01\) \\
    TR-iTRDH-PSB & \(1.25\)e\(+02\) & \(50\) & \(6.4\)e\(-03\) & \(129\) & \(71\) & \(3759\) & \(4.8\)e\(-01\) \\
    TR-TRDH-Andrei & \(1.25\)e\(+02\) & \(50\) & \(2.9\)e\(-03\) & \(170\) & \(91\) & \(8746\) & \(6.1\)e\(-01\) \\
    TR-iTRDH-Andrei & \(1.25\)e\(+02\) & \(50\) & \(6.8\)e\(-03\) & \(152\) & \(83\) & \(4809\) & \(6.1\)e\(-01\) \\
    TR-TRDH-Spec & \(1.25\)e\(+02\) & \(50\) & \(4.8\)e\(-03\) & \(126\) & \(69\) & \(4793\) & \(3.6\)e\(-01\) \\
    TR-iTRDH-Spec & \(1.25\)e\(+02\) & \(50\) & \(3.6\)e\(-03\) & \(162\) & \(82\) & \(2503\) & \(3.5\)e\(-01\) \\\hline\hline
  \end{tabular}
\end{table}

        \begin{figure}[ht]%
          \centering
          \includetikzgraphics[width = 0.48\linewidth]{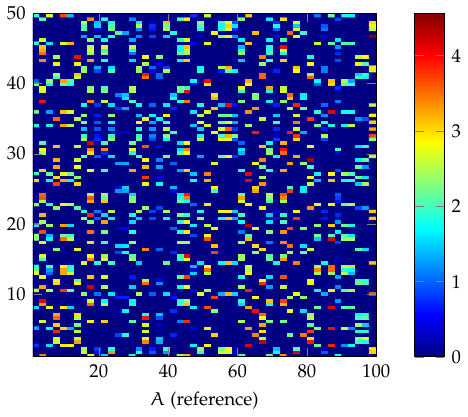}
          \includetikzgraphics[width = 0.48\linewidth]{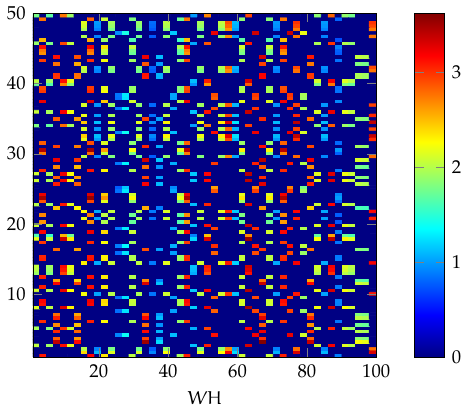}
          \includetikzgraphics[width = 0.48\linewidth]{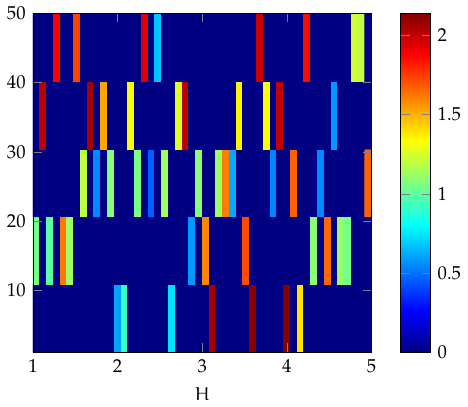}
          \includetikzgraphics[width = 0.48\linewidth]{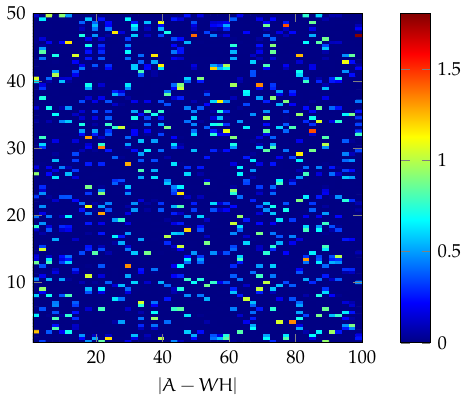}
          \caption{%
            \label{fig:nnmf-sol}
            Solution of~\eqref{eq:nnmf} with TR-TRDH-PSB\@.
          }
        \end{figure}

        We set \(m = 100\), \(n = 50\), \(k = 5\), \(\lambda = 10^{-1}\).
        The stopping tolerances \(\epsilon_a\) and \(\epsilon_r\) are set to \(10^{-5}\).
        \Cref{fig:nnmf-sol} shows the solution of~\eqref{eq:nnmf} with TR-TRDH-PSB\@.
        The statistics are reported in \Cref{tbl:nnmf}.
        We observe that R2 is trapped in a spurious stationary point, and that all TRDH and iTRDH solvers except TRDH-Spec and iTRDH-Spec reach their maximum number of iterations.
        All solvers using TR perform well with fewer objective and gradient evaluations than TR-R2.
        The number of proximal operator calls is lower for all TR-iTRDH variants.

      \subsection{Nonlinear support vector machine (SVM)}

        We now consider the nonlinear SVM described by \citep{aravkin-baraldi-orban-2022b} to classify digits of the MNIST dataset as either \(1\) or \(7\)---the other digits are removed.
        Let \(m\) be the number of images and \(n\) the vectorized image size.
        The problem reads
        \begin{equation}%
          \label{eq:nlin-svm}
          \minimize{x} \tfrac{1}{2} \|\textbf{1} - \tanh (b \odot (A^T x)) \|^2 + h(x),
        \end{equation}
        where \(A \in \R^{m \times n}\) is the data matrix, \(b\) is the vector of labels with values \(\pm 1\) for all its components, \(\odot\) denotes the elementwise product between two vectors, and \(h(x) = \lambda \|x\|_1\).
        We use \(n = 784\), \(m = 13007\) in the training set, \(m = 2163\) in the testing set, \(\lambda = 0.1\) and we initialize the problem at \(x = e\), the vector of ones.
        We set the absolute and relative stopping tolerances \(\epsilon_a\) and \(\epsilon_r\) to \(10^{-4}\).

        \begin{table}[ht]
  \centering
  \small
  \caption{%
  \label{tbl:svm}
  SVM~\eqref{eq:nlin-svm} statistics with $h = \lambda \|\cdot\|_1$.
  All variants of TR use an LBFGS Hessian approximation, and are given a maximum of \(100\) inner iterations.
  The train and test accuracies are the percentage of images correctly classified, computed by counting the number of elements of the residual \(\textbf{1} - \tanh (b \odot (A^T x))\) that are lower than \(1\) for the train and test problems respectively.
  }
  \begin{tabular}{rrrrrrrrr}
    \hline\hline
    solver & $f(x)$ & $h(x)/\lambda$ & $\sqrt{\xi / \nu}$ & (Train, Test) & $\#f$ & $\#\nabla f$ & $\#\prox{}$ & $t$ ($s$) \\\hline
    R2 & \(2.18\)e\(+02\) & \(2.4\)e\(+03\) & \(1.2\)e\(-01\) & \((99.3, 98.8)\) & \(265\) & \(199\) & \(265\) & \(3.7\)e\(+00\) \\
    TRDH-Spec & \(2.18\)e\(+02\) & \(2.4\)e\(+03\) & \(5.9\)e\(-02\) & \((99.3, 98.9)\) & \(306\) & \(191\) & \(611\) & \(4.4\)e\(+00\) \\
    iTRDH-Spec & \(2.18\)e\(+02\) & \(2.4\)e\(+03\) & \(5.9\)e\(-02\) & \((99.3, 98.9)\) & \(307\) & \(191\) & \(306\) & \(3.9\)e\(+00\) \\
    TRDH-PSB & \(2.38\)e\(+02\) & \(2.4\)e\(+03\) & \(4.5\)e\(-01\) & \((99.1, 98.8)\) & \(1001\) & \(686\) & \(2000\) & \(1.3\)e\(+01\) \\
    iTRDH-PSB & \(2.38\)e\(+02\) & \(2.4\)e\(+03\) & \(4.3\)e\(-01\) & \((99.1, 98.8)\) & \(1001\) & \(686\) & \(1000\) & \(1.3\)e\(+01\) \\
    TRDH-Andrei & \(3.00\)e\(+02\) & \(2.9\)e\(+03\) & \(1.2\)e\(+00\) & \((99.1, 99.0)\) & \(1001\) & \(418\) & \(2000\) & \(9.1\)e\(+00\) \\
    iTRDH-Andrei & \(3.00\)e\(+02\) & \(2.9\)e\(+03\) & \(2.7\)e\(+00\) & \((99.1, 99.0)\) & \(1001\) & \(418\) & \(1000\) & \(8.8\)e\(+00\) \\
    TR-R2 & \(2.18\)e\(+02\) & \(2.4\)e\(+03\) & \(4.7\)e\(-03\) & \((99.3, 98.8)\) & \(361\) & \(361\) & \(1512\) & \(5.3\)e\(+00\) \\
    TR-TRDH-PSB & \(2.18\)e\(+02\) & \(2.4\)e\(+03\) & \(4.7\)e\(-03\) & \((99.3, 98.8)\) & \(316\) & \(316\) & \(3882\) & \(6.5\)e\(+00\) \\
    TR-iTRDH-PSB & \(2.18\)e\(+02\) & \(2.4\)e\(+03\) & \(4.7\)e\(-03\) & \((99.3, 98.8)\) & \(377\) & \(377\) & \(2242\) & \(7.8\)e\(+00\) \\
    TR-TRDH-Andrei & \(2.18\)e\(+02\) & \(2.4\)e\(+03\) & \(4.7\)e\(-03\) & \((99.3, 98.9)\) & \(396\) & \(396\) & \(27155\) & \(1.4\)e\(+01\) \\
    TR-iTRDH-Andrei & \(2.18\)e\(+02\) & \(2.4\)e\(+03\) & \(4.5\)e\(-03\) & \((99.3, 98.9)\) & \(283\) & \(283\) & \(11418\) & \(1.2\)e\(+01\) \\
    TR-TRDH-Spec & \(2.18\)e\(+02\) & \(2.4\)e\(+03\) & \(4.7\)e\(-03\) & \((99.3, 98.8)\) & \(333\) & \(333\) & \(3390\) & \(7.5\)e\(+00\) \\
    TR-iTRDH-Spec & \(2.18\)e\(+02\) & \(2.4\)e\(+03\) & \(4.7\)e\(-03\) & \((99.3, 98.8)\) & \(472\) & \(472\) & \(1791\) & \(9.3\)e\(+00\) \\\hline\hline
  \end{tabular}
\end{table}

        \begin{figure}[ht]%
          \centering
          \includetikzgraphics[width = 0.48\linewidth]{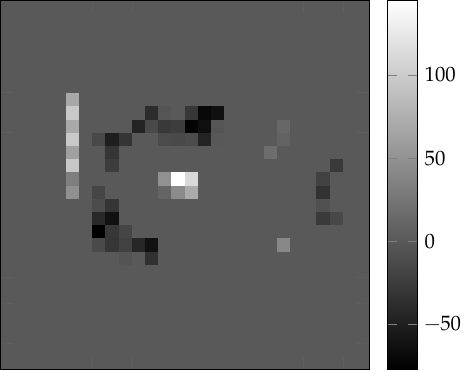}
          \hfill
          \begin{subfigure}{0.48\linewidth}
            \centering
            \includetikzgraphics[width = 0.4\linewidth]{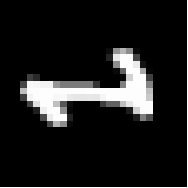}%
            \includetikzgraphics[width = 0.4\linewidth]{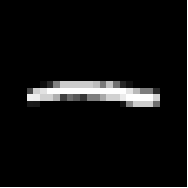}%
            \\
            \includetikzgraphics[width = 0.4\linewidth]{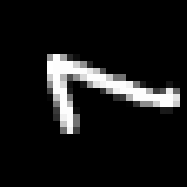}%
            \includetikzgraphics[width = 0.4\linewidth]{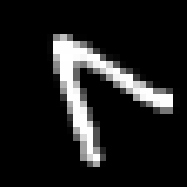}%
          \end{subfigure}
          \caption{%
            \label{fig:svm-sol}
            Solution map of~\eqref{eq:nlin-svm} with TR-TRDH-Spec (left), and sample 1's (right, top row) and 7's (right, bottom row) rotated digits from the MNIST dataset.
          }
        \end{figure}

        The solution map of~\eqref{eq:nlin-svm} representing the most important pixels to classify the images between \(1\) and \(7\) is shown in \Cref{fig:svm-sol} for the solver TR-TRDH-Spec, along with a few sample digits.
        We observe that the mid-height pixels are selected as the most important ones, which is consistent with the fact that the images show the digits sideways.
        The statistics are reported in \Cref{tbl:svm}, and show that TRDH-PSB, iTRDH-PSB, TRDH-Andrei and iTRDH-Andrei as main solvers exceed the maximum number of iterations.
        R2, TRDH-Spec and iTRDH-Spec are the most efficient and perform similar number of objective and gradient evaluations, with R2 terminating with a slightly higher criticality measure.
        TR-TRDH-PSB, TR-iTRDH-Andrei and TR-TRDH-Spec perform fewer objective and gradient evaluations, but more proximal operator evaluations than TR-R2.

      \subsection{FitzHugh-Nagumo inverse problem}

        Let \(v(x) = (v_1(x), \ldots, v_{n+1}(x))\) and \(w(x) = (w_1(x), \ldots, w_{n+1}(x))\) be sampled values of \(V(t; x)\) and \(W(t; x)\) for \(x \in \R^5\) satisfying the \citet{fitzhugh-1995} and \citet{nagumo-arimoto-1962} model for neuron activation
        \begin{equation}
          \label{eq:fh-nagumo}
          \frac{\mathrm{d} V}{\mathrm{d} t} = (V - V^3/3 - W + x_1) x_2^{-1}, \quad \frac{\mathrm{d} W}{\mathrm{d} t} = x_2 (x_3 V  - x_4 W + x_5).
        \end{equation}
        The samples are generated using a discretization of the time interval \(t \in [0, 20]\) with initial conditions \((V(0), W(0)) = (2, 0)\).
        We define a target solution that corresponds to a simulation of the \citet{van-der-pol-1926} oscillator by generating solutions \((\bar v (x), \bar w (x))\) of~\eqref{eq:fh-nagumo} with \(\bar x = (0, 0.2, 1, 0, 0)\), and we solve
        \begin{equation}%
          \label{eq:fh}
          \minimize{x} \tfrac{1}{2} \| (v(x) - \bar v (\bar x), w(x) - \bar w (\bar x) )\|_2^2 + h(x),
        \end{equation}
        where \(h(x) = \lambda \|x\|_0\) with \(n = 100\).
        The stopping tolerances \(\epsilon_a\) and \(\epsilon_r\) are set to \(10^{-4}\).
        We report the statistics of the solution of~\eqref{eq:fh} with \(\lambda = 10\) in \Cref{tbl:fh}.
        TRDH as main solver results in many objective and gradient evaluations compared to TR\@.
        Moreover, TR-TRDH-PSB, TR-iTRDH-PSB, TR-TRDH-Andrei and TR-iTRDH-Andrei perform fewer objective and gradient evaluations than TR-R2.
        TR-TRDH-Spec is the least efficient.

        \begin{table}[ht]
  \centering
  \small
  \caption{%
  \label{tbl:fh}
  FH~\eqref{eq:fh} statistics with $h = \lambda \|\cdot\|_0$.
  All variants of TR use an LBFGS Hessian approximation, and are given a maximum of \(200\) inner iterations.
  The optimal objective value is $f(x_\star) = 1.03e+00$.}
  \begin{tabular}{rrrrrrrrr}
    \hline\hline
    solver & $f(x)$ & $h(x)/\lambda$ & $\sqrt{\xi / \nu}$ & $\|x-x_\star\|_2$ & $\#f$ & $\#\nabla f$ & $\#\prox{}$ & $t$ ($s$) \\\hline
    TRDH-Spec & \(3.21\)e\(+01\) & \(5\) & \(7.0\)e\(+00\) & \(9.8\)e\(-01\) & \(501\) & \(351\) & \(1000\) & \(2.1\)e\(+00\) \\
    iTRDH-Spec & \(3.21\)e\(+01\) & \(5\) & \(5.9\)e\(+00\) & \(9.8\)e\(-01\) & \(501\) & \(351\) & \(500\) & \(2.5\)e\(+00\) \\
    TRDH-PSB & \(1.12\)e\(+00\) & \(2\) & \(2.7\)e\(+00\) & \(1.3\)e\(-01\) & \(501\) & \(401\) & \(1000\) & \(4.7\)e\(+00\) \\
    iTRDH-PSB & \(1.12\)e\(+00\) & \(2\) & \(2.4\)e\(+00\) & \(1.3\)e\(-01\) & \(501\) & \(401\) & \(500\) & \(4.7\)e\(+00\) \\
    TRDH-Andrei & \(1.15\)e\(+00\) & \(2\) & \(1.6\)e\(+00\) & \(1.5\)e\(-01\) & \(501\) & \(372\) & \(1000\) & \(4.2\)e\(+00\) \\
    iTRDH-Andrei & \(1.15\)e\(+00\) & \(2\) & \(2.8\)e\(+00\) & \(1.5\)e\(-01\) & \(501\) & \(372\) & \(500\) & \(5.3\)e\(+00\) \\
    TR-R2 & \(1.02\)e\(+00\) & \(2\) & \(5.0\)e\(-03\) & \(6.3\)e\(-03\) & \(327\) & \(236\) & \(36274\) & \(6.7\)e\(+00\) \\
    TR-TRDH-PSB & \(1.02\)e\(+00\) & \(2\) & \(5.7\)e\(-03\) & \(6.1\)e\(-03\) & \(183\) & \(149\) & \(56042\) & \(3.8\)e\(+00\) \\
    TR-iTRDH-PSB & \(1.02\)e\(+00\) & \(2\) & \(6.3\)e\(-03\) & \(6.2\)e\(-03\) & \(211\) & \(179\) & \(33452\) & \(3.7\)e\(+00\) \\
    TR-TRDH-Andrei & \(1.02\)e\(+00\) & \(2\) & \(5.6\)e\(-03\) & \(6.5\)e\(-03\) & \(185\) & \(138\) & \(49311\) & \(2.9\)e\(+00\) \\
    TR-iTRDH-Andrei & \(1.02\)e\(+00\) & \(2\) & \(6.4\)e\(-03\) & \(6.1\)e\(-03\) & \(168\) & \(151\) & \(23655\) & \(2.9\)e\(+00\) \\
    TR-TRDH-Spec & \(1.04\)e\(+00\) & \(2\) & \(1.0\)e\(+00\) & \(7.3\)e\(-02\) & \(501\) & \(428\) & \(109750\) & \(9.0\)e\(+00\) \\
    TR-iTRDH-Spec & \(1.02\)e\(+00\) & \(2\) & \(4.6\)e\(-03\) & \(1.1\)e\(-02\) & \(383\) & \(337\) & \(40354\) & \(6.1\)e\(+00\) \\\hline\hline
  \end{tabular}
\end{table}

        \begin{table}[ht]
          \centering
          \small
          \caption{%
            \label{tbl:fh-sol}
            FH~\eqref{eq:fh} (left) and constrained FH~\eqref{eq:fh-cstr} (right) solutions identified by the solvers tested.
            The unconstrained solution is \((0, 0.2, 1, 0, 0)\).
          }
          \begin{tabular}{r|rrrrr||rrrrr}
            \hline\hline
            & $x_1$ & $x_2$ & $x_3$ & $x_4$ & $x_5$
            & $x_1$ & $x_2$ & $x_3$ & $x_4$ & $x_5$ \\\hline
            TRDH-Spec & 0.32 & 0.43 & 0.83 & 0.77 & 0.44 & 0.00 & 0.50 & 0.54 & 0.00 & 0.00 \\
            iTRDH-Spec & 0.32 & 0.43 & 0.83 & 0.77 & 0.44 & 0.00 & 0.50 & 0.54 & 0.00 & 0.00 \\
            TRDH-PSB & 0.00 & 0.24 & 0.87 & 0.00 & 0.00 & 0.00 & 0.50 & 0.54 & 0.00 & 0.00 \\
            iTRDH-PSB & 0.00 & 0.24 & 0.87 & 0.00 & 0.00 & 0.00 & 0.50 & 0.54 & 0.00 & 0.00 \\
            TRDH-Andrei & 0.00 & 0.25 & 0.86 & 0.00 & 0.00 & 0.00 & 0.50 & 0.54 & 0.00 & 0.00 \\
            iTRDH-Andrei & 0.00 & 0.25 & 0.86 & 0.00 & 0.00 & 0.00 & 0.50 & 0.54 & 0.00 & 0.00 \\
            TR-R2 & 0.00 & 0.20 & 1.01 & 0.00 & 0.00 & 0.00 & 0.50 & 0.54 & 0.00 & 0.00 \\
            TR-TRDH-PSB & 0.00 & 0.20 & 1.01 & 0.00 & 0.00 & 0.00 & 0.50 & 0.54 & 0.00 & 0.00 \\
            TR-iTRDH-PSB & 0.00 & 0.20 & 1.01 & 0.00 & 0.00 & 0.00 & 0.50 & 0.54 & 0.00 & 0.00 \\
            TR-TRDH-Andrei & 0.00 & 0.20 & 1.01 & 0.00 & 0.00 & 0.00 & 0.50 & 0.54 & 0.00 & 0.00 \\
            TR-iTRDH-Andrei & 0.00 & 0.20 & 1.01 & 0.00 & 0.00 & 0.00 & 0.50 & 0.54 & 0.00 & 0.00 \\
            TR-TRDH-Spec & 0.00 & 0.18 & 1.07 & 0.00 & 0.00 & 0.00 & 0.50 & 0.54 & 0.00 & 0.00 \\
            TR-iTRDH-Spec & 0.00 & 0.20 & 1.01 & 0.00 & 0.00 & 0.00 & 0.50 & 0.54 & 0.00 & 0.00 \\\hline\hline
          \end{tabular}
        \end{table}

        \begin{figure}[ht]%
          \centering
          \includetikzgraphics[width = 0.48\linewidth]{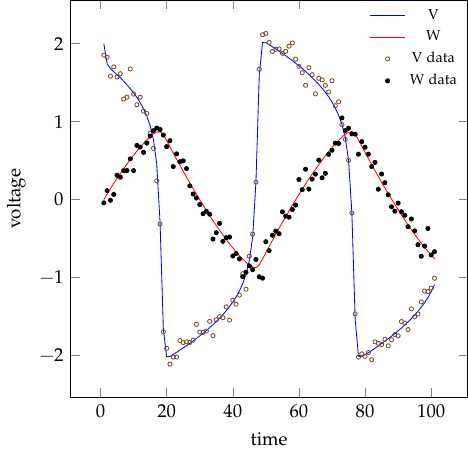}
          \hfill
          \includetikzgraphics[width = 0.48\linewidth]{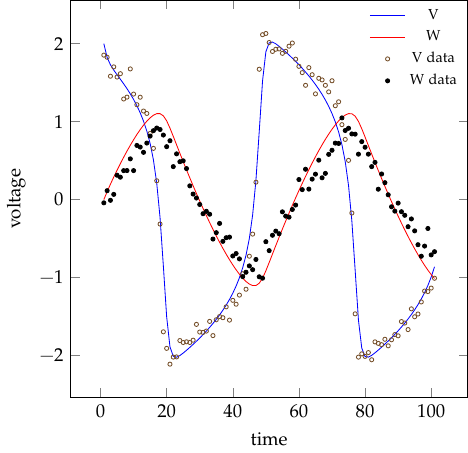}
          \caption{%
            \label{fig:fh}
            Solution of~\eqref{eq:fh} (left) and of~\eqref{eq:fh-cstr} (right) with TR-TRDH-PSB and sampled values of \(V\) and \(W\).
          }
        \end{figure}

        The left half of \Cref{tbl:fh-sol} reports the final solution identified by each solver.
        The solution of~\eqref{eq:fh} with TR-TRDH-PSB and the sampled values of \(V\) and \(W\) are displayed in \Cref{fig:fh}.
        We do not show results with R2 because it encountered numerical issues when solving the problem.

        We also solve the constrained variant
        \begin{equation}%
          \label{eq:fh-cstr}
          \minimize{x} \tfrac{1}{2} \| (v(x) - \bar v (\bar x), w(x) - \bar w (\bar x) )\|_2^2 + h(x), \quad \st \ x_2 \ge 0.5,
        \end{equation}
        and we keep all parameters the same, except for \(\lambda\) that we set to \(40\) to further enforce sparsity.

        \begin{table}[ht]
  \centering
  \small
  \caption{%
  \label{tbl:fh-cstr}
  Constrained FH~\eqref{eq:fh-cstr} statistics with $h = \lambda \|\cdot\|_1$.
  All variants of TR use an LBFGS Hessian approximation, and are given a maximum of \(200\) inner iterations.}
  \begin{tabular}{rrrrrrrrr}
    \hline\hline
    solver & $f(x)$ & $h(x)/\lambda$ & $\xi$ & $\|x-x_\star\|_2$ & $\#f$ & $\#\nabla f$ & $\#\prox{}$ & $t$ ($s$) \\\hline
    TRDH-Spec & \(4.43\)e\(+00\) & \(1.0\)e\(+00\) & \(1.2\)e\(-02\) & \(5.5\)e\(-01\) & \(327\) & \(226\) & \(653\) & \(2.6\)e\(+00\) \\
    iTRDH-Spec & \(4.43\)e\(+00\) & \(1.0\)e\(+00\) & \(1.1\)e\(-02\) & \(5.5\)e\(-01\) & \(328\) & \(226\) & \(327\) & \(2.4\)e\(+00\) \\
    TRDH-PSB & \(4.43\)e\(+00\) & \(1.0\)e\(+00\) & \(8.2\)e\(-03\) & \(5.5\)e\(-01\) & \(331\) & \(237\) & \(661\) & \(2.6\)e\(+00\) \\
    iTRDH-PSB & \(4.43\)e\(+00\) & \(1.0\)e\(+00\) & \(7.7\)e\(-03\) & \(5.5\)e\(-01\) & \(332\) & \(237\) & \(331\) & \(2.2\)e\(+00\) \\
    TRDH-Andrei & \(4.43\)e\(+00\) & \(1.0\)e\(+00\) & \(7.1\)e\(-03\) & \(5.5\)e\(-01\) & \(223\) & \(155\) & \(445\) & \(1.4\)e\(+00\) \\
    iTRDH-Andrei & \(4.43\)e\(+00\) & \(1.0\)e\(+00\) & \(7.1\)e\(-03\) & \(5.5\)e\(-01\) & \(224\) & \(155\) & \(223\) & \(1.6\)e\(+00\) \\
    TR-R2 & \(4.43\)e\(+00\) & \(1.0\)e\(+00\) & \(4.9\)e\(-03\) & \(5.5\)e\(-01\) & \(47\) & \(31\) & \(3057\) & \(4.4\)e\(-01\) \\
    TR-TRDH-PSB & \(4.43\)e\(+00\) & \(1.0\)e\(+00\) & \(2.7\)e\(-04\) & \(5.5\)e\(-01\) & \(38\) & \(29\) & \(6725\) & \(3.1\)e\(-01\) \\
    TR-iTRDH-PSB & \(4.43\)e\(+00\) & \(1.0\)e\(+00\) & \(7.9\)e\(-05\) & \(5.5\)e\(-01\) & \(36\) & \(27\) & \(3326\) & \(3.3\)e\(-01\) \\
    TR-TRDH-Andrei & \(4.43\)e\(+00\) & \(1.0\)e\(+00\) & \(8.7\)e\(-05\) & \(5.5\)e\(-01\) & \(39\) & \(30\) & \(7438\) & \(3.8\)e\(-01\) \\
    TR-iTRDH-Andrei & \(4.43\)e\(+00\) & \(1.0\)e\(+00\) & \(1.6\)e\(-03\) & \(5.5\)e\(-01\) & \(36\) & \(27\) & \(3117\) & \(3.4\)e\(-01\) \\
    TR-TRDH-Spec & \(4.43\)e\(+00\) & \(1.0\)e\(+00\) & \(4.1\)e\(-03\) & \(5.5\)e\(-01\) & \(36\) & \(27\) & \(5467\) & \(3.9\)e\(-01\) \\
    TR-iTRDH-Spec & \(4.43\)e\(+00\) & \(1.0\)e\(+00\) & \(1.8\)e\(-03\) & \(5.5\)e\(-01\) & \(38\) & \(25\) & \(3347\) & \(2.5\)e\(-01\) \\\hline\hline
  \end{tabular}
\end{table}

        The right half of \Cref{tbl:fh-sol} reports the solution and \Cref{tbl:fh-cstr} the statistics of the solve of~\eqref{eq:fh-cstr} with the tested solvers.
        Note that all solvers identify the same solution.
        We observe that all TR-TRDH and TR-iTRDH solvers perform fewer objective and gradient evaluations than TR-R2.
        However, TR-R2 has the lowest number of proximal operator calls out of all solvers using TR. 
        The right plot of \Cref{fig:fh} shows that the \(W\) part of the solution of~\eqref{eq:fh-cstr} with TR-TRDH-PSB does not match the data as tightly as that of~\eqref{eq:fh} in the left half of \Cref{fig:fh}.
        That is a consequence of enforcing \(x_2 \ge 0.5\), since the unconstrained solution verifies \(x_2 = 0.2\).

      \section{Discussion and future work}%
      \label{sec:discussion}

      The \(\ell_0\)- and \(\ell_1\)-norm regularizers are standard choices to promote sparsity.
      It is possible to derive the \(\iprox\) of other useful separable regularizers, including \(\ell_p\) pseudonorms to the \(p\)-th power, i.e.,
      \[
        h(x) = \|x\|_p^p = \sum_{i=1}^n |x_i|^p
        \quad (0 < p < 1),
      \]
      and those are also useful to promote sparsity \citep{cao-sun-xu-2013}.
      Much remains to be done, however, including deriving the \(\iprox\) of relevant non-separable regularizers, and studying other diagonal Hessian approximations than the ones considered above, including new diagonal quasi-Cauchy updates.

      TRDH performs well on the problems tested and is promising, but we were surprised to see the spectral gradient update often perform better than more sophisticated diagonal quasi-Newton updates, especially when using TRDH as main solver.
      Future research should seek to provide an explanation for that observation.
      In several instances, the variant of TRDH denoted ``iTRDH'' performs better than the basic version, which indicates that both algorithms are relevant.


      Finally, other solvers would likely benefit from using TRDH as a subproblem solver, including the methods for least-squares \(f\) of \citet{aravkin-baraldi-orban-2022b}.

      \small
      \subsection*{Acknowledgements}

      The authors thank Joshua Wolff from \'Ecole Normale Sup\'erieure des Techniques Avanc\'ees (ENSTA), Paris, for the work he conducted during his undergraduate internship  at GERAD in the summer of 2022 and that made this research possible.

      \bibliographystyle{abbrvnat}
      \bibliography{abbrv,indef-pg}

\begin{thebibliography}{30}
\providecommand{\natexlab}[1]{#1}
\providecommand{\url}[1]{\texttt{#1}}
\expandafter\ifx\csname urlstyle\endcsname\relax
  \providecommand{\doi}[1]{doi: #1}\else
  \providecommand{\doi}{doi: \begingroup \urlstyle{rm}\Url}\fi

\bibitem[Andrei(2019)]{andrei-2019}
N.~Andrei.
\newblock \doilink{10.1007/s11075-018-0562-7}{A diagonal quasi-{N}ewton
  updating method for unconstrained optimization}.
\newblock \emph{Numer. Algor.}, 81:\penalty0 575--–590, 2019.

\bibitem[Aravkin et~al.(2022{\natexlab{a}})Aravkin, Baraldi, and
  Orban]{aravkin-baraldi-orban-2022b}
A.~Aravkin, R.~Baraldi, and D.~Orban.
\newblock \doilink{10.13140/RG.2.2.28438.01604}{A {L}evenberg-{M}arquardt
  method for nonsmooth regularized least squares}.
\newblock Cahier du GERAD G-2023-58, GERAD, Montr\'eal, QC, Canada,
  2022{\natexlab{a}}.

\bibitem[Aravkin et~al.(2022{\natexlab{b}})Aravkin, Baraldi, and
  Orban]{aravkin-baraldi-orban-2022}
A.~Y. Aravkin, R.~Baraldi, and D.~Orban.
\newblock \doilink{10.1137/21M1409536}{A proximal quasi-{N}ewton trust-region
  method for nonsmooth regularized optimization}.
\newblock \emph{SIAM J. Optim.}, 32\penalty0 (2):\penalty0 900--929,
  2022{\natexlab{b}}.

\bibitem[Baraldi and Kouri(2022)]{baraldi-kouri-2022}
R.~Baraldi and D.~P. Kouri.
\newblock \doilink{10.1007/s10107-022-01915-3}{A proximal trust-region method
  for nonsmooth optimization with inexact function and gradient evaluations}.
\newblock \emph{Math. Program.}, 2022.

\bibitem[Baraldi and
  Orban(2022{\natexlab{a}})]{baraldi-orban-regularized-optimization-2022}
R.~Baraldi and D.~Orban.
\newblock \doilink{10.5281/zenodo.6940313}{{RegularizedOptimization.jl}:
  Algorithms for regularized optimization}.
\newblock
  \url{https://github.com/JuliaSmoothOptimizers/RegularizedOptimization.jl},
  February 2022{\natexlab{a}}.

\bibitem[Baraldi and
  Orban(2022{\natexlab{b}})]{baraldi-orban-regularized-problems-2022}
R.~Baraldi and D.~Orban.
\newblock \doilink{10.5281/zenodo.6940315}{{RegularizedProblems.jl}: Test cases
  for regularized optimization}.
\newblock
  \url{https://github.com/JuliaSmoothOptimizers/RegularizedProblems.jl},
  February 2022{\natexlab{b}}.

\bibitem[Baraldi and
  Orban(2022{\natexlab{c}})]{baraldi-orban-shifted-proximal-operators-2022}
R.~Baraldi and D.~Orban.
\newblock \doilink{10.5281/zenodo.6940317}{{ShiftedProximalOperators.jl}:
  Proximal operators for regularized optimization}.
\newblock
  \url{https://github.com/JuliaSmoothOptimizers/ShiftedProximalOperators.jl},
  February 2022{\natexlab{c}}.

\bibitem[Beck(2017)]{beck-2017}
A.~Beck.
\newblock \doilink{10.1137/1.9781611974997}{\emph{First-Order Methods in
  Optimization}}.
\newblock Number~25 in MOS-SIAM Series on Optimization. SIAM, Philadelphia,
  USA, 2017.

\bibitem[Becker and Fadili(2012)]{becker-fadili-2012}
S.~Becker and J.~Fadili.
\newblock
  \href{https://proceedings.neurips.cc/paper/2012/file/e034fb6b66aacc1d48f445ddfb08da98-Paper.pdf}{A
  quasi-{N}ewton proximal splitting method}.
\newblock In F.~Pereira, C.~Burges, L.~Bottou, and K.~Weinberger, editors,
  \emph{Advances in Neural Information Processing Systems}, volume~25. Curran
  Associates, Inc., 2012.

\bibitem[Becker et~al.(2019)Becker, Fadili, and Ochs]{becker-fadili-ochs-2019}
S.~Becker, J.~Fadili, and P.~Ochs.
\newblock \doilink{10.1137/18M1167152}{On quasi-{N}ewton forward-backward
  splitting: Proximal calculus and convergence}.
\newblock \emph{SIAM J. Optim.}, 29\penalty0 (4):\penalty0 2445--2481, 2019.

\bibitem[Birgin et~al.(2014)Birgin, Martínez, and
  Raydan]{birgin-martinez-raydan-2014}
E.~G. Birgin, J.~M. Martínez, and M.~Raydan.
\newblock \doilink{10.18637/jss.v060.i03}{Spectral projected gradient methods:
  Review and perspectives}.
\newblock \emph{Journal of Statistical Software}, 60\penalty0 (3):\penalty0
  1--21, 2014.

\bibitem[Cao et~al.(2013)Cao, Sun, and Xu]{cao-sun-xu-2013}
W.~Cao, J.~Sun, and Z.~Xu.
\newblock \doilink{10.1016/j.jvcir.2012.10.006}{Fast image deconvolution using
  closed-form thresholding formulas of \({L}_q\) (\(q = \tfrac{1}{2}\),
  \(\tfrac{2}{3}\)) regularization.}
\newblock \emph{J. Vis. Commun. Image R.}, 24\penalty0 (1):\penalty0 31--41,
  2013.

\bibitem[Cartis et~al.(2022)Cartis, Gould, and Toint]{cartis-gould-toint-2022}
C.~Cartis, N.~I.~M. Gould, and {\relax Ph}.~L. Toint.
\newblock \emph{Evaluation Complexity of Algorithms for Nonconvex
  Optimization}.
\newblock Number~30 in MOS-SIAM Series on Optimization. SIAM, Philadelphia,
  USA, 2022.

\bibitem[Conn et~al.(2000)Conn, Gould, and Toint]{conn-gould-toint-2000}
A.~R. Conn, N.~I.~M. Gould, and {\relax Ph}.~L. Toint.
\newblock \doilink{10.1137/1.9780898719857}{\emph{Trust-Region Methods}}.
\newblock Number~1 in MOS-SIAM Series on Optimization. SIAM, Philadelphia, USA,
  2000.

\bibitem[Dennis and Wolkowicz(1993)]{dennis-wolkowicz-1993}
J.~E. Dennis, Jr. and H.~Wolkowicz.
\newblock \doilink{10.1137/0730067}{Sizing and least-change secant methods}.
\newblock \emph{SIAM J. Numer. Anal.}, 30\penalty0 (5):\penalty0 1291--1314,
  1993.

\bibitem[Donoho(2006)]{donoho-2006}
D.~Donoho.
\newblock \doilink{10.1109/TIT.2006.871582}{Compressed sensing}.
\newblock \emph{IEEE T. Inform. Theory}, 52\penalty0 (4):\penalty0 1289--1306,
  2006.

\bibitem[Duchi et~al.(2011)Duchi, Hazan, and Singer]{duchi-hazan-2011}
J.~Duchi, E.~Hazan, and Y.~Singer.
\newblock \href{https://dl.acm.org/doi/10.5555/1953048.2021068}{Adaptive
  subgradient methods for online learning and stochastic optimization}.
\newblock \emph{J. Mach. Learn. Res.}, 12:\penalty0 2121–2159, 2011.

\bibitem[FitzHugh(1955)]{fitzhugh-1995}
R.~FitzHugh.
\newblock \doilink{10.1007/BF02477753}{Mathematical models of threshold
  phenomena in the nerve membrane}.
\newblock \emph{B. Math. Biophys.}, 17\penalty0 (4):\penalty0 257--278, 1955.

\bibitem[Gilbert and Lemaréchal(1989)]{gilbert-lemarechal-1989}
J.-C. Gilbert and C.~Lemaréchal.
\newblock \doilink{10.1007/BF01589113}{Some numerical experiments with
  variable-storage quasi-{N}ewton algorithms}.
\newblock \emph{Math. Program.}, 45:\penalty0 407--435, 1989.

\bibitem[Kim and Park(2008)]{kim-park-2008}
J.~Kim and H.~Park.
\newblock \href{http://hdl.handle.net/1853/20058}{Sparse nonnegative matrix
  factorization for clustering}.
\newblock Technical Report GT-CSE-08-01, Georgia Inst. of Technology, 2008.

\bibitem[Lions and Mercier(1979)]{lions-mercier-1979}
P.-L. Lions and B.~Mercier.
\newblock \doilink{10.1137/0716071}{Splitting algorithms for the sum of two
  nonlinear operators}.
\newblock \emph{SIAM J. Numer. Anal.}, 16\penalty0 (6):\penalty0 964--979,
  1979.

\bibitem[Lotfi et~al.(2020)Lotfi, {Bonniot de Ruisselet}, Orban, and
  Lodi]{lotfi-bonniot-orban-lodi-2020}
S.~Lotfi, T.~{Bonniot de Ruisselet}, D.~Orban, and A.~Lodi.
\newblock \doilink{10.13140/RG.2.2.27851.41765/1}{Stochastic damped {L-BFGS}
  with controlled norm of the {H}essian approximation}.
\newblock 2020.
\newblock OPT2020 Conference on Optimization for Machine Learning.

\bibitem[Nagumo et~al.(1962)Nagumo, Arimoto, and
  Yoshizawa]{nagumo-arimoto-1962}
J.~Nagumo, S.~Arimoto, and S.~Yoshizawa.
\newblock \doilink{10.1109/JRPROC.1962.288235}{An active pulse transmission
  line simulating nerve axon}.
\newblock \emph{Proceedings of the IRE}, 50\penalty0 (10):\penalty0 2061--2070,
  1962.

\bibitem[Nazareth(1995)]{nazareth-1995}
J.~L. Nazareth.
\newblock If quasi-{N}ewton then why not quasi-{C}auchy?
\newblock \emph{SIAG/OPT Views-and-News}, 6:\penalty0 11--14, 1995.

\bibitem[Orban et~al.(2020)Orban, Siqueira, and
  {contributors}]{orban-siqueira-linearoperators-2020}
D.~Orban, A.~S. Siqueira, and {contributors}.
\newblock \doilink{10.5281/zenodo.2559295}{{LinearOperators.jl}}.
\newblock \url{https://github.com/JuliaSmoothOptimizers/LinearOperators.jl},
  September 2020.

\bibitem[Rockafellar and Wets(1998)]{rockafellar-wets-1998}
R.~Rockafellar and R.~Wets.
\newblock \doilink{10.1007/978-3-642-02431-3}{\emph{Variational Analysis}},
  volume 317.
\newblock Springer Verlag, 1998.

\bibitem[Scheinberg and Tang(2016)]{scheinberg-tang-2016}
K.~Scheinberg and X.~Tang.
\newblock \doilink{10.1007/s10107-016-0997-3}{Practical inexact proximal
  quasi-{N}ewton method with global complexity analysis}.
\newblock \emph{Math. Program.}, \penalty0 (160):\penalty0 495--529, 2016.

\bibitem[Tibshirani(1996)]{tibshirani-1996}
R.~Tibshirani.
\newblock \doilink{10.1111/j.2517-6161.1996.tb02080.x}{Regression shrinkage and
  selection via the lasso}.
\newblock \emph{J. Roy. Statist. Soc. Ser. B}, 58\penalty0 (1):\penalty0
  267--288, 1996.

\bibitem[van~der Pol(1926)]{van-der-pol-1926}
B.~van~der Pol.
\newblock \doilink{10.1080/14786442608564127}{{{LXXXVIII}}. {{On}}
  “relaxation-oscillations”}.
\newblock \emph{The London, Edinburgh, and Dublin Philosophical Magazine and
  Journal of Science}, 2\penalty0 (11):\penalty0 978--992, 1926.

\bibitem[Zhu et~al.(1999)Zhu, Nazareth, and
  Wolkowicz]{zhu-nazareth-wolkowicz-1999}
M.~Zhu, J.~L. Nazareth, and H.~Wolkowicz.
\newblock \doilink{10.1137/S1052623498331793}{The quasi-{C}auchy relation and
  diagonal updating}.
\newblock \emph{SIAM J. Optim.}, 9\penalty0 (4):\penalty0 1192--1204, 1999.

\end{thebibliography}

  \end{document}